\documentclass[12pt]{article}
\usepackage[a4paper, hmargin=2.199cm, vmargin={3cm, 3cm}]{geometry}
\usepackage{graphicx, array, blindtext}
\usepackage{wrapfig}
\usepackage{amsmath,amsthm}
\usepackage{hyperref}
\usepackage[font=footnotesize]{caption}
\usepackage[justification=centering]{caption}
\usepackage[margin=1.5cm]{caption}

\usepackage{mwe}
\usepackage{enumitem}
\usepackage{amsfonts}
\usepackage{tikz}
\usepackage{mathrsfs}
\usepackage{graphicx}
\usepackage{picinpar}
\usepackage{amssymb}
\usepackage{mathtools}
\usepackage{csquotes}

\newcommand \eps {\varepsilon}

\newcommand \C {\mathbb C}
\newcommand \R {\mathbb R}

\newcommand \Z {\mathbb Z}

\newcommand \D {\mathbb D}

\newcommand \inv {^{-1}}
\newcommand \nin{\not\in}

\newcommand* \bb [1]{\left({#1}\right)}

\newcommand* \bset[1] {\left\{{#1}\right\}}
\newcommand* \indic [1]{\mathbf 1_{#1}}

\newcommand* \norm[2]{\left|\left|{#1}\right|\right|_{#2}}
\newcommand* \limit [2]{\underset{{#1}\rightarrow{#2}}{\lim}\;}
\newcommand* \bunion[3]{\bigcup_{{#1} = {#2}}^{#3}}
\newcommand* \bintersect[3]{\bigcap_{{#1} = {#2}}^{#3}}
\newcommand* \sumit [3]{\underset{{#1} = {#2}}{\overset{#3}\sum}\;}

\newcommand* \limitsup [2]{\underset{{#1}\rightarrow{#2}}{\limsup}\;}

\newcommand* \abs[1] {\left|{#1}\right|}

\newtheorem{thm}{Theorem}[section]
\newtheorem{lem}[thm]{Lemma}
\newtheorem{claim}[thm]{Claim}

\newtheorem{defn}[thm]{Definition}

\newtheorem{rmk}[thm]{Remark}

\newtheorem{quest}[thm]{Question}
\newtheorem*{thm*}{Theorem}
\newtheorem*{lem*}{Lemma}
\newtheorem*{claim*}{Claim}
\newtheorem*{prop*}{Proposition}
\newtheorem*{cor*}{Corollary}
\newtheorem*{defn*}{Definition}
\newtheorem*{examp*}{Example}
\newtheorem*{rmk*}{Remark}
\newtheorem*{remind*}{Reminder}
\newtheorem*{quest*}{Question}

\begin{document}
\title{Measurably entire functions and their growth}
\author{Adi Gl\"ucksam\thanks{Supported in part by ERC Advanced Grant~692616 and ISF Grant~382/15.}}
\maketitle

\begin{abstract}
In 1997 B. Weiss introduced the notion of measurably entire functions and proved that they exist on every arbitrary free $\C$-action defined on a standard probability space. In the same paper he asked about the minimal possible growth rate of such functions. In this work we show that for every arbitrary free $\C$-action defined on a standard probability space there exists a measurably entire function whose growth rate does not exceed $\exp\bb{\exp [\log^p |z|]}$ for any $p>3$. This complements a recent result by Buhovski, Gl\"ucksam, Logunov, and Sodin who showed that such functions cannot have a growth rate smaller than $\exp\bb{\exp [\log^p |z|]}$ for any $p<2$.
\end{abstract}

\section{Introduction}
A measure space $(X,\mathcal B,\mu)$ is called a \underline{standard probability space} if $\mu(X)=1$ and there exists a topology $\tau$ such that $(X,\tau)$ is metrizable as a topological space, $\mathcal B$ is the completion of the $\sigma$-algebra generated by the open sets of $\tau$, and for every $\eps>0$ there exists a compact set $K$ such that $\mu(K)>1-\eps$.

Let $(X,\mathcal B,\mu)$ be a standard probability space. A map $f:X\rightarrow X$ is called \underline{probability} \underline{preserving} if for every $B\in\mathcal B$, $\mu(B)=\mu(f\inv(B))$. We denote by $PPT(X)$ the group of all invertible probability preserving transformations from $(X,\mathcal B,\mu)$ to itself. We use the standard topology on this group, defined by the pull back of the weak operator topology restricted to unitary operators on $L_2(X,\mathcal B,\mu)$ by the Koopman representation associated with the action,  $T\mapsto U_Tf$, where $[U_Tf](x)=f(Tx)$ (see \cite[Page~61]{Halmos1960}).

A \underline{probability preserving action of $\C$} (a $\C$-action in short) is a continuous homomorphism $T:\C\rightarrow PPT(X)$. A $\C$-action $T:\C\rightarrow PPT(X)$ is called \underline{free} if for $\mu$-almost every $x\in X$, $T_zx=x$ implies that $z=0$. In other words, there are no periodic points almost surely.

Let $\mathcal E$ denote the space of entire functions endowed with the local uniform topology, and let $\mathcal B$ denote the Borel structure associated with it. The complex plane acts on $(\mathcal E,\mathcal B)$ by translations defined by $\bb{T_w f}(z)=f(z+w)$.

Whether there exists a probability measure $\lambda$ defined on $(\mathcal E,\mathcal B)$ such that $T$ is a $\C$-action on $(\mathcal E,\mathcal B,\lambda)$ is not a trivial fact. In fact, it was not known until Weiss showed such measures exist using notions from dynamical systems, which we shall introduce now:
\begin{defn}
Let $(X,\mathcal B,\mu)$ be a standard probability space, and suppose $T:\C\rightarrow PPT(X)$ is  a $\C$-action. A map $F:X\rightarrow\C$ is called \underline{measurably entire} if it is a non-constant measurable function and for $\mu$-almost every $x\in X$ the map $F_x:\C\rightarrow\C$ defined by $F_x(z):=F(T_zx)$ is entire.
\end{defn}
The existence of measurably entire functions is closely related to the question of existence of translation invariant random entire functions. On one hand, the space of entire functions, $\mathcal E$, endowed with the topology of local uniform convergence is a Polish space, and so the existence of a translation invariant probability measure on $\mathcal E$ is an example of a measurably entire function. On the other hand, the existence of a measurably entire function produces a translation invariant random entire function by defining the measure
$$
\mu_F(A):=\mu\bb{\bset{x\in X;\; F_x\in A}},\; A\subset\mathcal E ,\text {measurable}.
$$

Some years ago Mackey asked the following question:
\begin{quest}[Mackey]
Does every probability preserving free action of $\C$ on a standard probability space admits a measurably entire function?
\end{quest}
Weiss answered Mackey's question in 1997:
\begin{thm}[Weiss 1997,\;\cite{Weiss1996}]
For every free probability preserving action of $\C$ on a standard probability space there exists a measurably entire function.
\end{thm}
Weiss' paper gives rise to an abundance of measurably entire functions and in particular answers Mackey's question positively. In his paper Weiss raised several questions, one of them was about the possible growth rate of such functions, measured by the asymptotic growth of the function $M_f(R):=\underset{z\in\overline{R\D}}\max\;\abs{f(z)}$, where $\overline{R\D}:=\bset{\abs z\le R}$. There are two possible interpretations for this question:
\begin{enumerate}[label=(\roman*),leftmargin=0.7cm]
\item What is the minimal growth rate of a measurably entire function of a $\C$-action on a standard probability space $(X,\mathcal B,\mu)$?
\item Given a $\C$-action on a standard probability space $(X,\mathcal B,\mu)$, what is the minimal growth rate of a measurably entire function?
\end{enumerate}
We recently proved in a joint work with L. Buhovsky, A. Logunov, and M. Sodin the following theorem, which gives an almost full answer to the first interpretation. We state this theorem using the terminology of measurably entire functions, where $\log^\alpha x:=\bb{\log x}^\alpha$.
\begin{thm}\cite[Theorem~1]{Us2017}\label{thm:Us}
\begin{enumerate}[label=(\alph*),leftmargin=0.7cm]
\item There exists a standard probability space $(X,\mathcal B,\mu)$ with a free $\C$-action, $T$, for which there exists a measurably entire function $F$ such that for $\mu$ almost every $x\in X$, and for every $\eps>0$:
$$
\limitsup R\infty \frac{\log\log \underset{z\in \overline{R\D}}\max\;\abs{F(T_zx)}}{\log^{2+\eps} R}=0.
$$
\item For every standard probability space $(X,\mathcal B,\mu)$ for every measurably entire function $F:X\rightarrow\C$ $\mu$-almost every $x$, either $z\mapsto F(T_zx)$ is a constant function or for every $\eps>0$
$$\limit R\infty\frac{\log\log \underset{z\in \overline{R\D}}\max\;\abs{F(T_zx)}}{\log^{2-\eps} R}=\infty.$$
\end{enumerate}
\end{thm}
While Weiss' paper tells us such functions always exist, part (b) of Theorem \ref{thm:Us} gives a lower bound for the minimal possible growth rate of measurably entire functions defined for a general free $\C$-action defined on a standard probability space, but not an upper bound.

We would like to emphasize the difference between the two interpretations. While in the first interpretation one may choose the measure space (and therefore the action) as well as the measurably entire function, in the second one the action is given to us, and one may only choose the measurably entire function.

In this paper we will construct a measurably entire function with bounded growth rate for a general free action:
\begin{thm}\label{thm:UprLwrBnd}
Let $(X,\mathcal B,\mu)$ be a standard probability space, and suppose $T:\C\rightarrow PPT(X)$ is a free action. Then there exists a measurably entire function $F:X\rightarrow \C$ such that for $\mu$-almost every $x\in X$ for every $\eps>0$:
\begin{eqnarray}\label{eq:UprBnd}
\limit R\infty\frac{\log\log\underset{z\in \overline{R\D}}\max\;\abs{F(T_zx)}}{\log^{3+\eps}R}=0.
\end{eqnarray}
\end{thm}
This theorem gives an upper bound for the minimal growth rate of measurably entire functions defined on a general free $\C$-action. Nevertheless, note that there is still a gap between the lower and upper bounds known to us so far:
\begin{quest}
Is the gap between the lower bound given by Theorem \ref{thm:Us} and upper bound given by Theorem \ref{thm:UprLwrBnd} justified? Namely, does there exists a $\C$-action on a standard probability space $(X,\mathcal B,\mu)$ and $p\in(2,3)$ such that for every measurably entire function $F:X\rightarrow\C$ for $\mu$-almost every $x\in X$:
\begin{eqnarray*}
\limit R\infty\frac{\log\log\underset{z\in \overline{R\D}}\max\;\abs{F(T_zx)}}{\log^pR}=\infty.
\end{eqnarray*}
\end{quest}
\subsection{Notation}
Given $a>0$ we denote by $S_a$ the square centered at the origin of edge length $2a$, namely $S_a=[-a,a]^2$. \\
Let $A\subset\C$ and $\omega\in\C$. We define by $A(\omega):=\omega+A$, the translation of the set $A$ by $\omega$. \\
For a set $\Omega\subset\C$ we define the sets
$$
\Omega^{+\eps}:=\bset{z\in\C,\; d(z,\Omega)<\eps}\;\;,\;\;\Omega^{-\eps}:=\bset{z\in\C,\; d(z,\Omega^c)>\eps}.
$$

\subsection{Acknowledgments}
The author would like to thank her PhD adviser, Mikhail Sodin, for acquainting her with this most intriguing question, conducting many interesting discussions, and reading many revisions of this paper with great patience. The author is grateful to Jon Aaronson for several helpful discussions, and Alon Nishry for insightful editorial remarks. Last but not least, the author would like to thank the referee for a careful review, and for useful comments and corrections.
\section{Preliminary Lemmas:}
\subsection{Complex Analysis lemmas}
In this subsection we will state and prove lemmas using tools from complex analysis. Throughout this section we will use the letters $\lambda$ and $\mu$ to denote elements of $\C$ (and not measures). The first lemma, proven in this subsection, is a lemma that creates a non-negative subharmonic function with `windows', i.e rectangles where $v= 0$.
\begin{lem}\label{lem:PrisonBreak}
For every $C\ge1$ and for every set $\Lambda\subset\C$, such that for every $\lambda\neq\mu\in\Lambda,\; \norm{\mu-\lambda}\infty>2$, there exists a subharmonic function $v$ such that:
\begin{enumerate}[label=($P_\arabic*$),leftmargin=0.7cm]
\item For every $\lambda\in\Lambda$, define $D_\lambda:=\bset{v= 0}\cap S_1(\lambda)$, then $D_\lambda^{+\frac1C}\subset S_1(\lambda)$ while $S_1(\lambda)\setminus D_\lambda$ is a union of at most $20$ rectangles of edge length at most $2$ and edge width at most $\frac2C$. In particular, $\frac{m\bb{D_\lambda}}{m\bb{S_1}}\ge 1-\frac {80}C$.
\item For every $z\in\C$, $v(z)\le\exp\bb{2\pi C}\exp\bb{\frac{\pi C}2\abs z}$.
\item For every $\lambda\in\Lambda$, $v|_{D_\lambda^{+\frac 5{3C}}\setminus D_\lambda^{+\frac 1{3C}}}\ge\frac12$.
\end{enumerate}
\end{lem}
\begin{proof}
Given $C\ge1$ we define the subharmonic function
$$
b_C(z)=b_C(x+iy)=	\begin{cases}
						\cos\bb{\frac{\pi C}2\cdot y}\cosh\bb{\frac{\pi C}2\cdot x} & , \abs y<\frac1C\\
						0&, \text{otherwise}
						\end{cases}.					
$$
This function is $0$ outside an infinite horizontal strip of width $\frac2C$. Given $\lambda\in\Lambda$ we define the \underline{window function assigned to $\lambda$} by
$$
v_\lambda(z):=\max\bset{b_C( {i}z-\lambda+1),b_C(z-\lambda+i),b_C( {i}z-\lambda-1),b_C(z-\lambda-i)}.
$$

The set $\bset{z,\; v_\lambda(z)\neq 0}$ looks like a window, whose cornices have `infinite tails' (see Figure \ref{fig:windowFunction}). In addition, note that $v_\lambda|_{S_1^{-\frac1C}(\lambda)}=0$, while $v_\lambda|_{S_3(\lambda)}\le e^{\frac{3\pi C}2}$.

\begin{figure}[!ht]
\centering
    \includegraphics[width=0.65\textwidth]{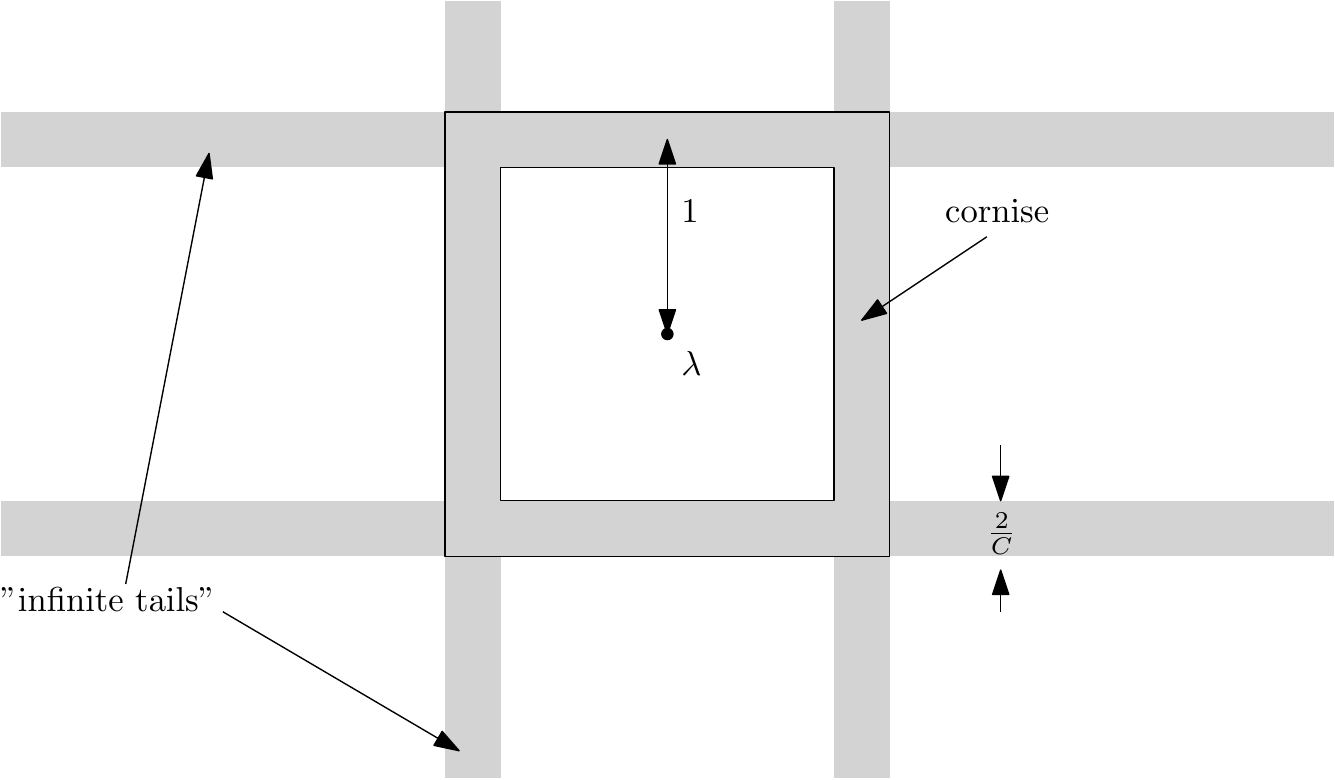}
  \caption{The gray area is where $v_\lambda\neq 0$.}
  \label{fig:windowFunction}
  \end{figure}

We would like to take maximum over window functions assigned to $\lambda\in\Lambda$. Formally, we would like to define $v(z)=\sup\bset{v_\lambda(z),\lambda\in\Lambda}$. The problem is it is not clear that locally we take supremum over a finite set, and even if we do, the elements of $\Lambda$ are not necessarily aligned in the sense that $S_1^{-\frac1C}(\lambda)$ may intersect `infinite tails' of many elements $\mu\neq\lambda\in\Lambda$. We get that the `infinite tails' of windows that were created for different elements of $\Lambda$ might intrude into the window area of other elements, and the number of `intruders' is not necessarily bounded and might cover the whole window, making property $(P_1)$ impossible to satisfy. To overcome this problem, we create a grid using the same base function $b_C$, and then take maximum over `window functions' assigned only to elements inside each grid component, bounding the number of possible `intruders' in each window.
 
Formally, we define the sets $\Z_{odd}:=\bset{2n+1,\; n\in\Z},\; \Z_{even}:=\bset{2n,\; n\in\Z}$ and define the function
$$
v_0(z):=e^{2\pi C}\max\bset{b_C(i\omega+z),b_C(\omega+iz);\; \omega\in\Z_{odd}}.
$$
For every $\omega\in\Z_{odd}$ fixed for every $z\in\C$, 
\begin{eqnarray*}
b_C(i\omega+z)\le\exp\bb{\frac{\pi C}2\cdot \abs{Re(z)}}\\
b_C(\omega+iz)\le\exp\bb{\frac{\pi C}2\cdot \abs{Im(z)}},
\end{eqnarray*}
independently of $\omega$. We get that $v_0$ is bounded by $\exp\bb{\frac{\pi C}2\max\bset{\abs{Re(z)},\abs{Im(z)}}+2\pi C}$. In addition, locally this function is a maximum of at most two subharmonic functions, and therefore it is subharmonic (see Figure \ref{fig:grid}).\\
\begin{figure}[!ht]
\centering
    \includegraphics[width=0.5\textwidth]{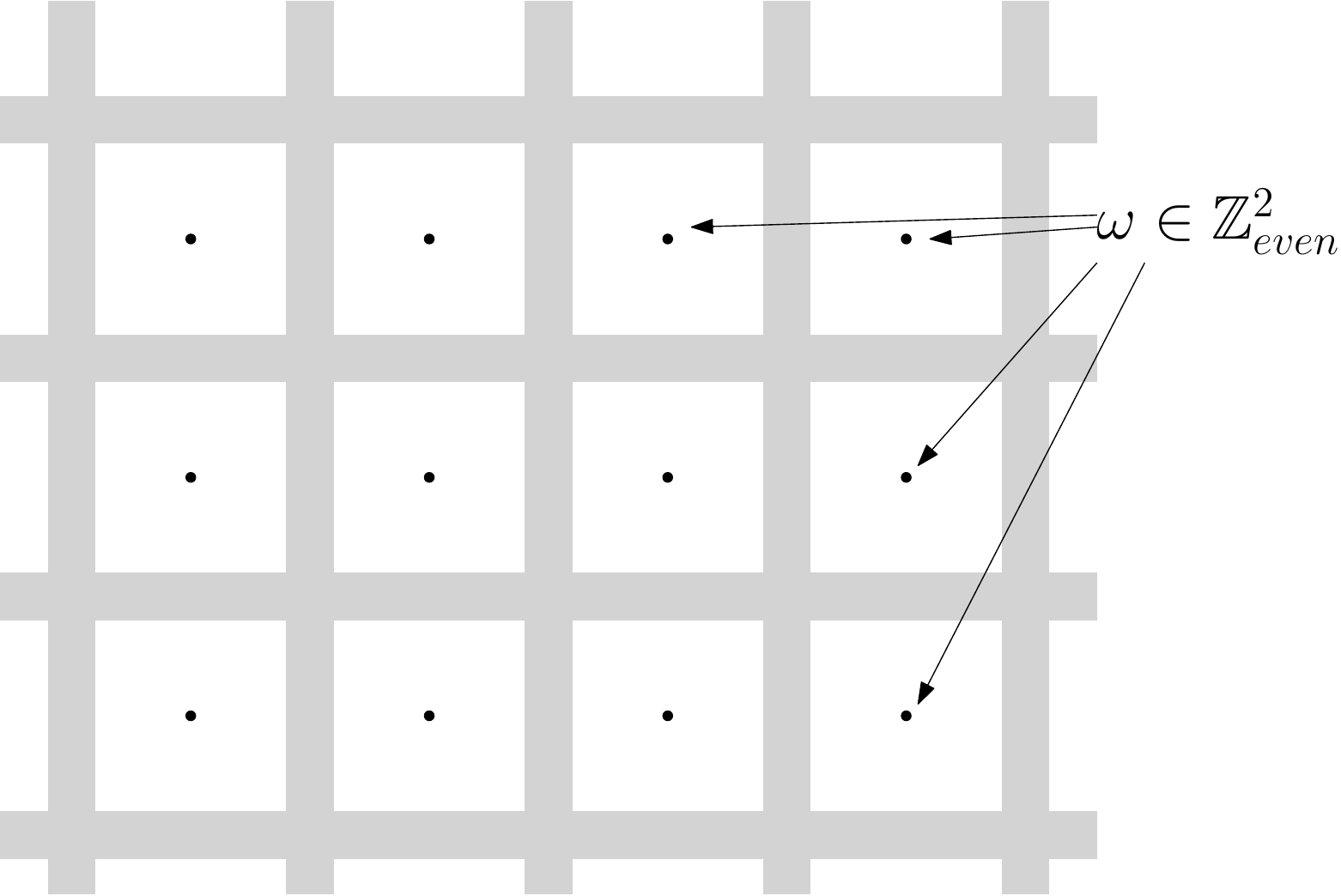}
  \caption{The grid: The gray area is the set where $\bset{v_0\neq 0}$ while the white area is the set where $\bset{v_0= 0}$.}
  \label{fig:grid}
\end{figure}

For every $\lambda\in\Lambda$ we define the set
$$
A_\lambda:=\bset{\omega\in\Z_{even}^2, S_1(\lambda)\cap S_1(\omega)\neq\emptyset}.
$$

$A_\lambda$ is the set of elements $\omega\in\Z_{even}^2$ such that a square of edge $2$ centered at $\omega$ intersects a square of edge $2$ centered at $\lambda$. For every $\lambda\in\Lambda$, $\# A_\lambda\le 4$, since the squares are disjoint, aligned, and have the same edge length, and therefore every such intersection creates a rectangle such that at least one of its corners belongs to $S_1(\lambda)$ (see Figure \ref{fig:NumOfElemsInA_lambda}). 

Symmetrically, for every $\omega\in\Z_{even}^2$ the set defined by
$$
B^\omega:=\bset{\lambda\in\Lambda, S_1(\lambda)\cap S_1(\omega)\neq\emptyset}
$$
also contains at most 4 elements. 

\begin{figure}[!ht]
\centering
  \begin{center}
    \includegraphics[width=0.5\textwidth]{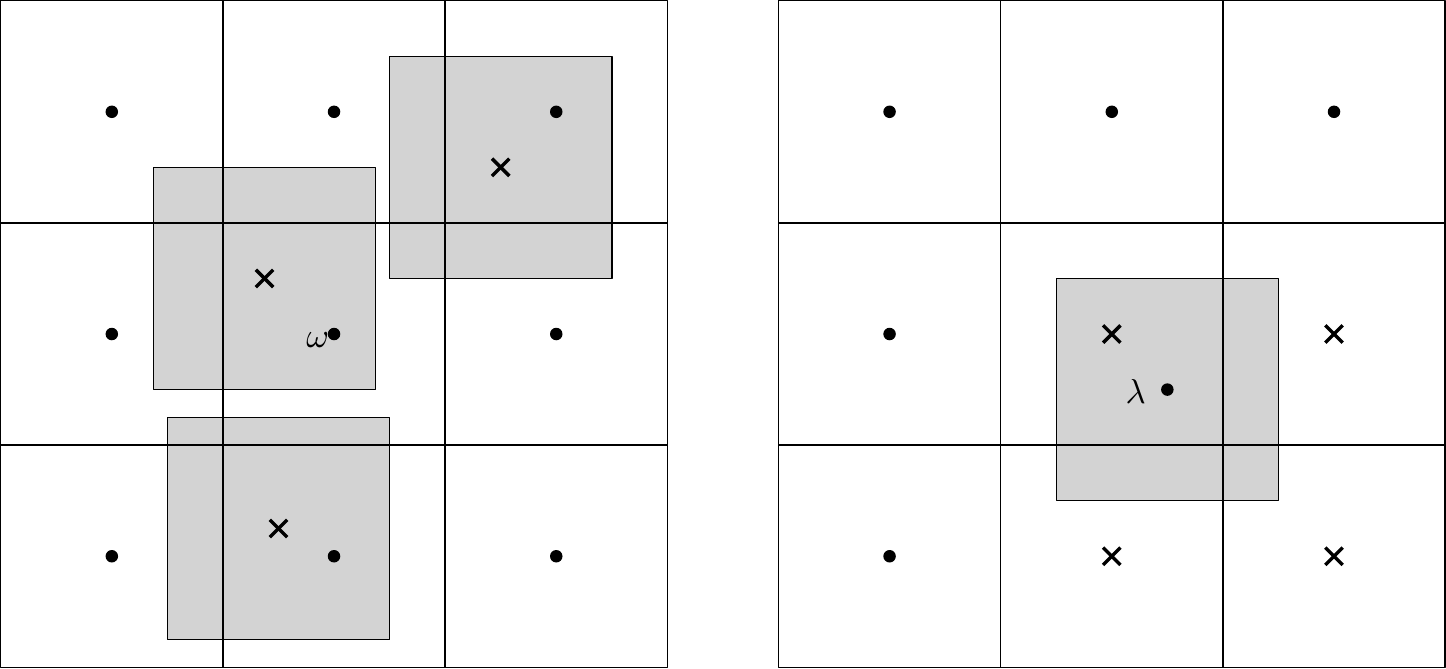}
  \end{center}
  \caption{The right hand picture: the points marked by x represent elements of $\Z_{even}^2$ which belong to $A_\lambda$. The left hand picture: the points marked by x represent elements of $\Lambda$ which belong to $B^\omega$.}
  \label{fig:A_lambdaB_omega}
  \end{figure}

\begin{figure}[!ht]
\centering
  \begin{center}
    \includegraphics[width=0.2\textwidth]{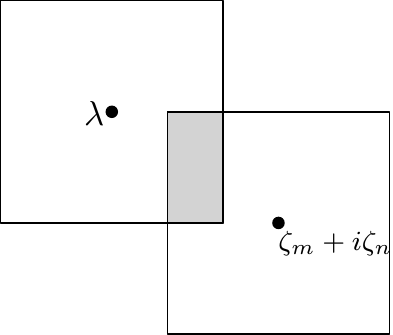}
  \end{center}
  \caption{Since the squares are aligned and have the same edge length, every intersection creates a rectangle such that at least one of its corners belongs to $S_1(\lambda)$.}
  \label{fig:NumOfElemsInA_lambda}
\end{figure}

As mentioned before, for every $\lambda\in\Lambda$, $v_\lambda$ is subharmonic, and $v_\lambda|_{S_1^{-\frac1C}(\lambda)}=0$, while $v_\lambda|_{S_3(\lambda)}\le e^{\frac{3\pi C}2}$. Define
$$
v(z):= \max\bset{v_0(z),\underset{\lambda\in B^\omega}\max\; v_\lambda(z)},\;\omega\in\Z_{even}^2,\; z\in S_1(\omega).
$$
We will first show that this function is well defined and subharmonic. For every $\omega\in\Z_{even}^2$, and every $z\in S_1^{+\frac2{3C}}(\omega)\setminus S_1^{-\frac 2{3C}}(\omega)$
$$
v_0(z)\ge \cos\bb{\frac{\pi C}2\cdot\frac2{3C}}\exp\bb{\frac{\pi C}2\cdot\min\bset{\abs{Re(z)},\abs{Im(z)}}+2\pi C}\ge \frac12 \exp\bb{2\pi C}.
$$
Since $v_\lambda|_{S_3(\lambda)}\le e^{\frac{3\pi C}2}$ and $C\ge 1$ we get that for every $\omega\in\Z_{even}^2$ and every $\lambda\in B^\omega$,
$$
v_0|_{S_1(\omega)\setminus S_1^{-\frac2{3C}}(\omega)}\ge \frac{e^{2\pi C}}2\ge e^{\frac{3\pi C}2}\ge v_\lambda|_{S_3(\lambda)}.
$$
In particular, $v$ defined above, is well defined and is subharmonic since locally it is a maximum over a finite set of subharmonic functions. Moreover, note that for every $\mu\nin B^\omega$ the function $v_\mu$ does not effect the definition of $v$ in $S_1(\omega)$ in any way.\\

Next, for every $\lambda\in\Lambda$ we look at the set $D_\lambda:=\bset{z\in\C,\; v(z)= 0}\cap S_1(\lambda)$. Note that $D_\lambda$ is in fact $S_1(\lambda)$ once we remove from it strips of width $\frac2C$ that originated in the base function $b_C$. By the way $b_C$ was defined, $b_C(x+iy)\ge\frac12$ if $\abs y\le\frac2{3C}$. We get that if $z\in D_\lambda^{+\frac5{3C}}\setminus D_\lambda^{+\frac1{3C}}$, then $z$ belongs to a translation and/or rotation of the strip $\abs y\le\frac2{3C}$, and since $v$ is defined as a maximum of such functions, in particular $v(z)\ge \frac12$.

It is left to bound the number of `intruders' for every $\lambda\in\Lambda$, or formally the number of copies of the set $\bset{b_C\neq 0}$ intersecting $S_1(\lambda)$. For this it is enough to bound the number of elements in $\underset{\omega\in A_\lambda}\bigcup B^\omega\setminus\bset\lambda$. Why is this enough? As we saw above for every $\mu\in\Lambda$ outside the set $\underset{\omega\in A_\lambda}\bigcup S_1(\omega)$, the definition of the function $v\left|_{\scriptscriptstyle{\underset{\omega\in A_\lambda}\bigcup S_1(\omega)}}\right.$ is unchanged whether $\mu\in\Lambda$ or not, and in particular if $A_\mu\cap A_\lambda=\emptyset$, then whether $\mu\in\Lambda$ or not does not change the way $v$ is defined inside $S_1(\lambda)$. We conclude that it is enough to bound the number of elements in the set $\underset{\omega\in A_\lambda}\bigcup B^\omega\setminus\bset\lambda$, but the later is bounded by $16$ as the number of elements in $A_\lambda$ is at most $4$ and the number of elements in $B^\omega$ is at most $4$ as well. Adding the 4 rectangles created by $v_\lambda$ itself we get 20 `intruding' rectangles as needed. Note that though every intersection of $S_1(\lambda)$ with $S_1(\omega)$ contributes two potentially `intruding' rectangles, one horizontal and one vertical, only one of them can intersect $S_1(\mu)^{-\frac1C}$ for $S_1(\mu)$ intersecting $S_1(\omega)$. The reason is that $S_1(\lambda)$ and $S_1(\mu)$ are disjoint, and so one can be positioned either to the left/right with respect to the other (thus intersecting the horizontal rectangle) or above/bellow (thus intersecting the vertical rectangle), but not both.
\end{proof}

An application of this lemma allows us to `glue' together several subharmonic functions $\bset{u_\lambda}$ restricted to disjoint compact subsets of $\C$, $S_1(\lambda)$:
\begin{lem}\label{lem:SHExtension}
Let $C>7$, and let $\Lambda\subset\C$ be such that for every $\lambda\neq\mu\in\Lambda ,\; \norm{\mu-\lambda}\infty>2$. Assume that for every $\lambda\in\Lambda $ there exists $u_\lambda:\C\rightarrow[0,\infty)$ subharmonic such that for a positive constant $\mathcal M$
$$
\underset{\lambda\in\Lambda}\max\;\underset{z\in S_1}\max\; u_\lambda(z)\le \mathcal M.
$$
Then there exists a subharmonic function $u$ such that:
\begin{enumerate}[label=($SH_\arabic*$),leftmargin=1.2cm]
\item For every $\lambda\in\Lambda$ there exists a set $D_\lambda$ such that $D_\lambda^{+\frac1C}\subset S_1(\lambda)$ while $S_1(\lambda)\setminus D_\lambda$ is contained in a union of at most $20$ rectangles of edge length at most $2$ and edge width at most $\frac2C$, and for every $z\in D_\lambda$ we have $u(z)=  u_\lambda(z-\lambda)$.
\item 
$$
\underset{z\in S_C}\max\; u(z)\le 2\mathcal Me^{\pi C^2}.
$$
\item For every $\lambda\in\Lambda$
$$
\underset{z \in  D_\lambda^{+\frac5{3C}}\setminus  D_\lambda^{+\frac 1{3C}}}\min\; u(z)\ge \mathcal M.
$$
\end{enumerate}
\end{lem}
\begin{proof}
Let $v$ denote the subharmonic function obtained by Lemma \ref{lem:PrisonBreak} with the set $\Lambda$ and the constant $C$. Define for every $\lambda\in\Lambda$ the set $D_\lambda:=\bset{v=0}\cap S_1(\lambda)$. Following property $(P_1)$ of the function $v$ guaranteed by Lemma \ref{lem:PrisonBreak}, this set satisfies all the properties described in $(SH_1)$. Define the function
$$
u(z)=		\begin{cases}
		\max\bset{2\mathcal M\cdot v\bb{z}, u_\lambda(z-\lambda)}&, z\in D_\lambda^{+\frac1{3C}}, \;\lambda\in\Lambda\\
		2\mathcal M\cdot v\bb{z}&, \text{otherwise}.
		\end{cases}
$$
We will first show that $u$ is subharmonic. Fix $\lambda\in\Lambda$. Following property $(P_3)$ of the function $v$, for every $z\in D_\lambda^{+\frac {5}{3C}}\setminus D_\lambda^{+\frac 1{3C}}$ we have $v\bb{z}\ge\frac12$, while $\underset{z\in S_1}\max\; u_\lambda(z)\le \mathcal M$. And so:
\begin{eqnarray*}
\underset{z\in  D_\lambda^{+\frac {5}{3C}}\setminus  D_\lambda^{+\frac 1{3C}}}\min\;  2\mathcal M\cdot v\bb{z}\ge\frac{2\mathcal M}2= \mathcal M\ge \underset{z\in S_1}\max\; u_\lambda,
\end{eqnarray*}
which implies that $u$ defined above is well defined and subharmonic as locally it is a maximum between two subharmonic functions. This also proves property $(SH_3)$.

To see that property $(SH_1)$ holds, note that since $u_\lambda\ge 0$, for every $z\in D_\lambda$ we have $u_\lambda(z)=u(z-\lambda)$ as needed.

To see that property $(SH_2)$ holds we observe that for $C>7$,
\begin{eqnarray*}
\underset{z\in S_C}\max\; u&=&2\mathcal M\cdot \underset{z\in S_C}\max\; v\le 2\mathcal M\cdot \exp\bb{2\pi C}\exp\bb{\frac{\pi C}2\cdot \underset{z\in S_C}\max\;\abs z}\le 2\mathcal Me^{\pi C^2},
\end{eqnarray*}
concluding our proof.
\end{proof}
The next lemma is an extension of the previous lemma for `glueing' several entire functions.
\begin{lem}\label{lem:EntireExtension}
Let $C$ and $B$ be sufficiently large constants and let $\Lambda\subset S_C$ be such that for every $\lambda\neq\mu\in\Lambda,\; \norm{\mu-\lambda}\infty>2$. Assume that for every $\lambda\in\Lambda$ there exists $f_\lambda$ analytic in $S_1$ such that for some $\mathcal M>40\log C$,
$$
\underset{\lambda\in\Lambda}\max\;\underset{z\in S_1}\max\; \abs{f_\lambda(z)}\le \exp\bb{2^{1-B}\mathcal M}.
$$
Then there exists an entire function $f$ with the following properties:
\begin{enumerate}[label=($E_{\arabic*}$),leftmargin=0.9cm]
\item For every $\lambda\in\Lambda$ define the set 
$$
A_\lambda=S_1(\lambda)\cap\bset{z,\; \abs{f(z)-f_\lambda(z-\lambda)}<\exp\bb{-\frac{\mathcal M}4}}.
$$
Then for every $\eps>0$, $m\bb{S_1(\lambda)\setminus A_\lambda^{-\eps}}=O\bb{\frac1C+\eps}$\footnote{In fact, $O\bb{\frac1C+\eps}=\frac{160}C+200\eps$.}.
\item 
$$
\underset{z\in S_C}\max\; \abs{f(z)}\le \exp\bb{2^{1-B}\mathcal M\cdot e^{\pi C^2}}.
$$
\end{enumerate}
\end{lem}
\begin{proof}
Let $u$ be the subharmonic function constructed in Lemma \ref{lem:SHExtension} with the set $\Lambda$, the constant $C$, and the functions $u_\lambda=\log_+\abs{f_\lambda}$. Recall the sets $D_\lambda\subset S_1(\lambda)$ defined for every $\lambda\in\Lambda$ in Lemma \ref{lem:SHExtension}. These sets were defined so that for every $z\in D_\lambda$, $u(z)=  u_\lambda(z-\lambda)$, and $S_1(\lambda)\setminus D_\lambda$ is a union of at most 20 rectangle of edge length at most $2$ and edge width at most $\frac2C$. Let $\chi:\C\rightarrow[0,1]$ be a smooth function with the following properties:
\begin{enumerate}[label=(\alph*),leftmargin=0.7cm]
\item For every $\lambda\in\Lambda$, $\chi|_{D_\lambda^{+\frac 1{4C}}}= 1$.
\item $\chi|_{\C\setminus \underset{\lambda\in\Lambda}\bigcup  D_\lambda^{+\frac {3}{4C}}}= 0$.
\item For every $z\in\C$, $\abs{\nabla\chi(z)}\le 100C$.
\end{enumerate}

For example, by taking a convolution of the normalization of a rescaling of the bump function
$$
\phi(z)=\begin{cases}
		\exp\bb{-\frac1{1-\abs z^2}}&, \abs z\le 1\\
		0&, \text{otherwise}
		\end{cases}
$$
by $\frac1{4C}$ so that its integral is one, with the function
$$
\psi(z)=\underset{\lambda\in\Lambda}\sum\; \indic{D_\lambda^{+\frac1{2C}}(\lambda)}(z).
$$

\begin{figure}[!ht]
\centering
  \begin{center}
    \frame{\includegraphics[width=0.5\textwidth]{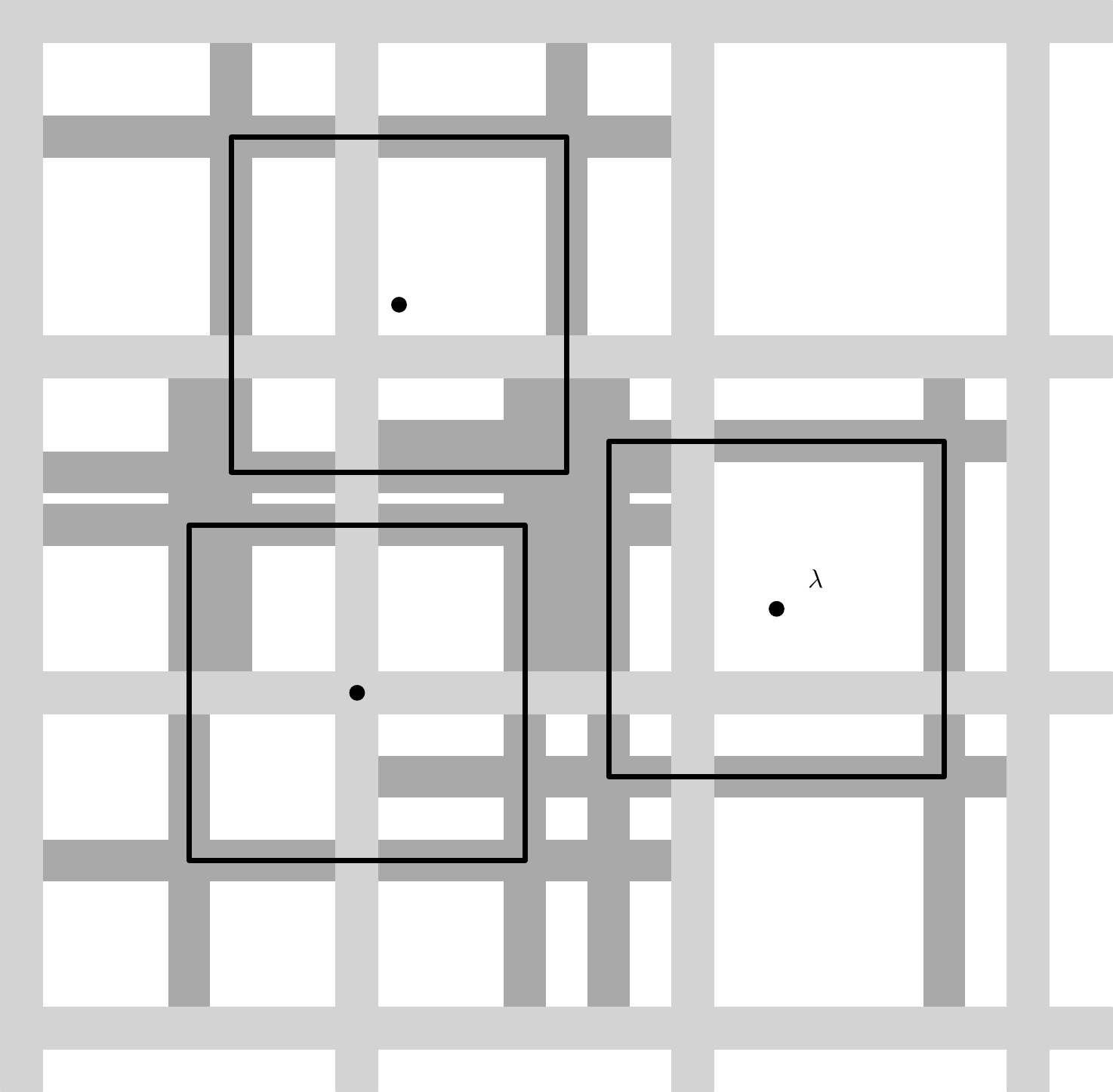}
    \includegraphics[width=0.5\textwidth]{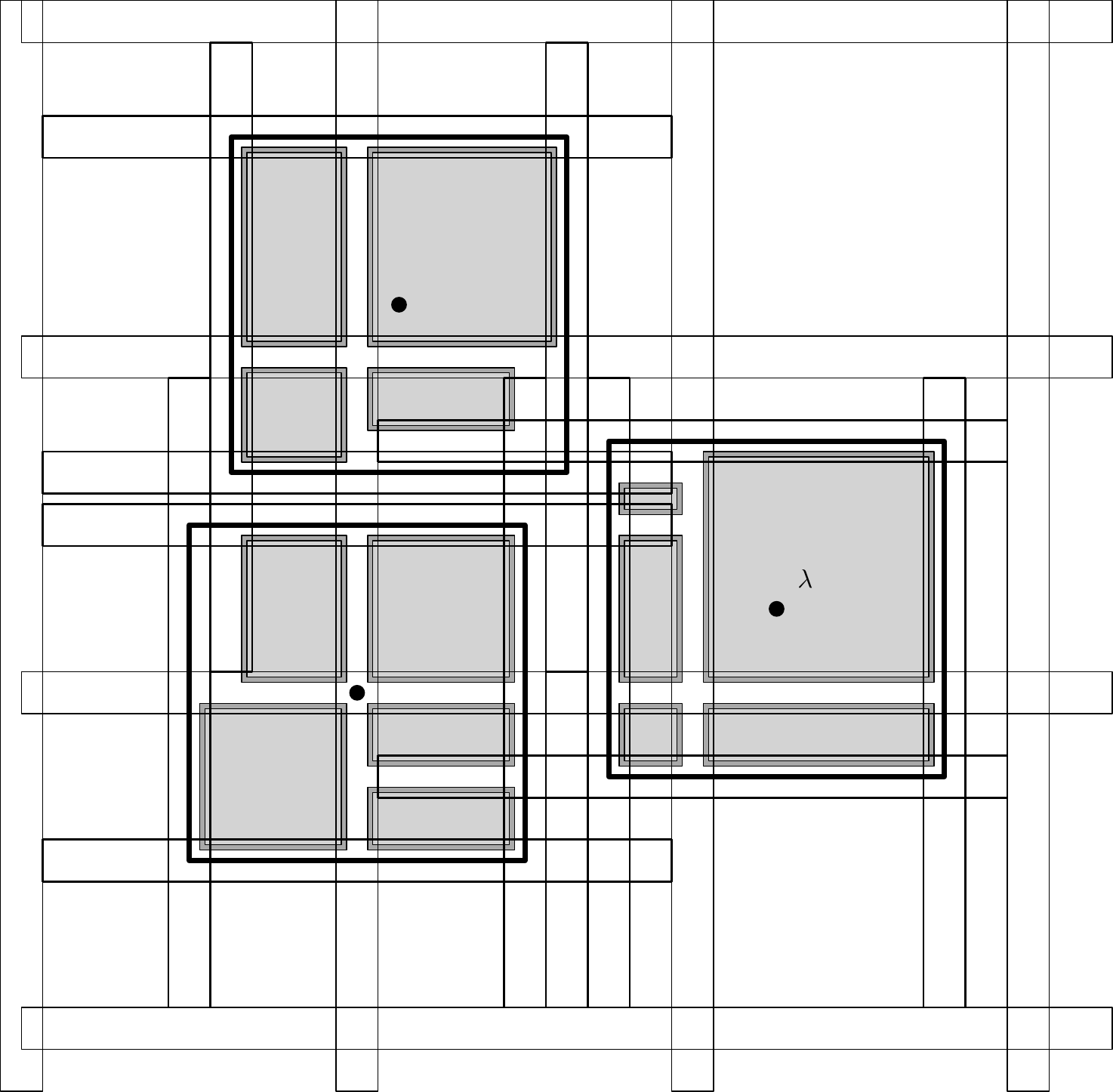}
    }
  \end{center}
  \caption{On the picture to the left, the white area is the area where $\bset{u=0}$. The corridors created by the grid are colored in light gray, while the ones created by elements of $\Lambda$ are colored in dark gray. For every element of $\Lambda$ in the picture, the sets $D_\lambda$ is the white areas within the relevant square.\\
The picture to the right is the same picture, but the light gray area now describes the set where $\chi\equiv 1$, the white area describes the set where $\chi\equiv 0$ and the dark gray area describes the set where the transition occurs.
  }
  \label{fig:chi}
  \end{figure}

Define the function $g_0:\underset{\lambda\in\Lambda}\bigcup S_1(\lambda)\rightarrow\C$ by
$$
g_0(z):= \underset{\lambda\in\Lambda}\sum\; f_\lambda(z-\lambda)\cdot\indic{S_1(\lambda)}(z).
$$
As $D_\lambda^{+\frac1C}\subset S_1(\lambda)$ and the collection $\bset{S_1(\lambda)}_{\lambda\in\Lambda}$ is a collection of disjoint squares, $g_0$ is holomorphic where it is defined, as locally it is just $f_\lambda$ for one particular $\lambda\in\Lambda$.
Define
$$
g(z)=g_0(z)\cdot\chi(z).
$$
Note that $g$ is well defined as the area where $\chi= 0$ separates $S_1(\lambda)$ from $S_1(\mu)$, for $\lambda\neq\mu$ (see Figure \ref{fig:chi}). Next we define the entire function
$$
f(z)=g(z)-\alpha(z),
$$
where $\alpha$ is H\"ormander's solution~\cite[Theorem~4.2.1]{Hormander} to the $\bar\partial$-equation

$$
\bar\partial g(z)=\bar\partial\chi(z)\cdot g_0(z)=\bar\partial \alpha(z),
$$
satisfying
$$
\underset{\C}\int\abs{\alpha(z)}^2\frac{e^{-u(z)}}{\bb{\abs z^2+1}^2}dz\le\frac12\underset{\C}\int\abs{\bar\partial g(z)}^2e^{-u(z)}dz.
$$
First of all, let us bound the right hand side of this inequality: by the definition of $\chi$ and property $(SH_3)$ of $u$-
\begin{eqnarray}\label{eq:RHS}
\underset{\C}\int\abs{\bar\partial g(z)}^2e^{-u(z)}dz&=&\underset{\lambda\in\Lambda}\sum \underset{D_\lambda^{+\frac {3}{4C}}\setminus D_\lambda^{+\frac 1{4C}}}\int\abs{\bar\partial g(z)}^2e^{-u(z)}dz\nonumber\\
&\le& \underset{z\in\C}\max\;\abs{\nabla \chi(z)}^2\exp\bb{2^{2-B}\mathcal M}\cdot e^{-\mathcal M}\cdot m\bb{\underset{\lambda\in\Lambda}\bigcup D_\lambda^{+\frac {3}{4C}}\setminus D_\lambda^{+\frac 1{4C}}}\nonumber\\
&\le&10^4C^2\exp\bb{\mathcal M\bb{2^{2-B}-1}}\cdot 4C^2\cdot\frac{40\cdot 2}{2C}\le C^4\exp\bb{-\frac{\mathcal M}2}\nonumber\\
&\Rightarrow& \underset{\C}\int\abs{\bar\partial g(z)}^2e^{-u(z)}dz\le C^4\exp\bb{-\frac{\mathcal M}2},
\end{eqnarray}
provided that $B$, $C$, and $\mathcal M$ are large enough.

Next, let us find an upper bound for $f$ on $S_C$: Fix $c_0=\frac1{16C}<\frac 1{4C}$, then by Cauchy's integral formula:
\begin{eqnarray*}
\abs{f(z)}^2&\le&\frac1{\pi c_0^2}\underset{B(z,c_0)}\int\abs{f(w)}^2dm(w)\le\frac2{\pi c_0^2}\underset{B(z,c_0)}\int\abs{g(z)}^2+\abs{\alpha(w)}^2dm(w)=I_1+I_2.
\end{eqnarray*}
To bound $I_1$ note that by the way $g$ is defined
$$
I_1\le 2\;\underset{S_{C}}\max\;\abs{g}^2\le2\; {\underset{\lambda\in\Lambda}\max\;}\underset{S_1}\max\; \abs{f_\lambda}^2\le 2\exp\bb{2^{2-B}\mathcal M}\le\frac12\exp\bb{2^{2-B}\mathcal M\cdot e^{\pi C^2}}.
$$
On the other hand, using (\ref{eq:RHS}) and property $(SH_2)$ of $u$ in Lemma \ref{lem:SHExtension},
\begin{eqnarray*}
I_2&=&\frac2{\pi c_0^2}\underset{B(z,c_0)}\int\abs{\alpha(w)}^2dm(w)\le \frac4{\pi c_0^2}\exp\bb{\underset{z\in S_C}\max\; u}C^4\underset{\C}\int \abs{\alpha(w)}^2\frac{e^{-u(w)}}{\bb{\abs w^2+1}^2}dm(w)\\
&\le&\frac {64C^2}{\pi}\exp\bb{2^{2-B}\mathcal Me^{\pi C^2}}C^8\cdot \exp\bb{-\frac{\mathcal M}2}\le C^{11}\exp\bb{2^{2-B}\mathcal M\cdot e^{\pi C^2}}\exp\bb{-\frac12\mathcal M}\\
&\le&\frac12\exp\bb{2^{2-B}\mathcal M\cdot e^{\pi C^2}},
\end{eqnarray*}
provided that $\mathcal M$ is big enough so that $e^{-\frac{\mathcal M}2}\cdot C^{11}<\frac12$. Combining the two estimates we get that
\begin{eqnarray*}
\abs{f(z)}&\le&\sqrt{I_1+I_2}< \exp\bb{2^{1-B}\mathcal Me^{\pi C^2}}.
\end{eqnarray*}
We conclude that property $(E_2)$ holds.

Finally, to see property $(E_1)$, note that:
\begin{enumerate}[label=(\alph*),leftmargin=0.7cm]
\item For every $z\in D_\lambda^{-\frac1{4C}}$, $B(z,c_0)\subset D_\lambda$, which implies that $T_\lambda u(z)=u_{\lambda}(z)$ by property $(SH_1)$ of $u$, and therefore $\underset{w\in B(z,c_0)}\max\; e^u\le \exp\bb{2^{1-B}\mathcal M}$.
\item By the way $f$ was defined, for every $w\in D_\lambda^{+\frac1{4C}}$ we have $f(w)= f_\lambda(w-\lambda)-\alpha(w)$.
\end{enumerate}
By Cauchy's integral formula applied to $z\in D_\lambda^{-\frac1{4C}}$, and the function $\bb{f(z)-T_{-\lambda}f_\lambda(z)}$ which is holomorphic in $B(z,c_0)$, and by using the bound given by (\ref{eq:RHS})
\begin{eqnarray*}
&&\abs{f(z)-T_{-\lambda}f_\lambda(z)}^2=\abs{\frac1{\pi c_0^2}\underset{B(z,c_0)}\int \bb{f(w)-T_{-\lambda}f_\lambda(w)} dm(w)}^2\overset{(b)}=\abs{\frac1{\pi c_0^2}\underset{B(z,c_0)}\int \alpha(w) dm(w)}^2\\
&\le&\frac1{\pi c_0^2}\underset{B(z,c_0)}\int\abs{\alpha(w)}^2dm(w)\le C^7 \underset{w\in\D(z,c_0)}\max\; e^u\cdot\underset{\C}\int\abs{\alpha(z)}^2\frac{e^{-u(z)}}{\bb{\abs z^2+1}^2}dm(z)\\
&\overset{(a)}\le&C^7\cdot \exp\bb{2^{1-B}\mathcal M}\cdot C^4\exp\bb{-\frac{\mathcal M}2}=C^{11}\cdot \exp\bb{\mathcal M\bb{2^{1-B}-\frac12}}\le\exp\bb{-\frac{\mathcal M}4}
\end{eqnarray*}
for $B$, $C$, and $\mathcal M$ large enough. We obtain that $D_\lambda^{-\frac1{4C}}\subset A_\lambda\Rightarrow D_\lambda^{-\frac1{4C}-\eps}\subset A_\lambda^{-\eps}$, but since $S_1(\lambda)\setminus D_\lambda$ is a union of at most 20 rectangles, we get by the inclusion of the sets that
$$
m(S_1(\lambda)\setminus A_\lambda^{-\eps})\le m\bb{S_1(\lambda)\setminus D_\lambda^{-\frac1{4C}-\eps}}\le 40\cdot\bb{5\eps+\frac4C}=O\bb{\eps+\frac1C},
$$
concluding our proof.
\end{proof}
\subsection{Measure Theoretic Lemmas}

In this subsection we will present lemmas related to ergodic theory and dynamics.

\begin{defn}
Let $(X,\mathcal B,\mu)$ be a standard probability space, and suppose $T:\C\rightarrow PPT(X)$ is a free $\C$-action. Let $S\subset\C$ be a compact set. A set $B\in\mathcal B$ is called  an \underline{$S$-set} if 
\begin{enumerate}[label=($F_{\arabic*}$),leftmargin=0.7cm]
\item For every $z\neq w\in S$, $T_zB\cap T_wB=\emptyset$.
\item For every $B'\subset B\subset X$ measurable, and every $A\subset S\subset\C$ measurable, the set $AB':=\underset{z\in A}\bigcup T_zB'$ is a measurable subset of $X$.
\end{enumerate}
\end{defn}

In the definition above, the set $S$ is the set of `shifts' by the action of $\C$, marked $T$, while $B\subset X$ is a very small set that we `shift' by elements of $S$ in the space $X$. For our purpose, $S$ will be a two dimensional square.

We are interested in $S$-sets, for $S\subset\C$ a square, because these sets allow us to assign for every $x\in B$ a function $f_x:S\rightarrow \C$, which is holomorphic in $S$, creating a measurably entire function $f:SB\rightarrow\C$ defined by $f(T_zx)=f(z,x)=f_x(z)$ without worrying about inconsistencies in the definition of $f$. Note that as $B$ is an $S$-set, the map $T_zx\mapsto (z,x)$ is well defined, and so is our function $f$. We would therefore like to approximate our space $X$ by a sequence of sets $\bset{S_{a_n}B_n}_{n=1}^\infty$ where $B_n$ is an $S_{a_n}$-set, and $a_n\nearrow\infty$.

\begin{rmk}\label{rmk:JonsObservation}
For every square $S\subset\C$ and every $B\in\mathcal B$, an $S$-set
$$
\mu\bb{AB}:=\mu\bb{\bset{T_zB,\; z\in A}}=\frac{m(A)}{m(S)}\cdot\mu\bb{SB},\;\text{ for every } A\subset S\text{ measurable}.
$$
\end{rmk}
\underline{Explanation:} Because the measure $\mu$ is a translation invariant measure, we may assume without loss of generality that $S=[-a,a]^2$ for some $a>0$. Every such cube $S\subset\C$ is also a topological group, with the group action defined by
$$
\tau_wz=(w+z) \;mod\; a.
$$
This group is a Polish group (i.e it is a separable completely metrizable topological space), and by Haar's theorem there exists a unique measure (up to multiplication by constants) which is invariant under the group's action. In this case, this is just Lebesgue's measure restricted to $S$. For every $S$-set, $B$, define the measure $\nu_B$ on measurable subsets of $S$ by $\nu_B(A):=\mu(AB)$. Since $B$ is an $S$-set, this is indeed a well defined measure. In addition, since $\mu$ is translation invariant, we get that this measure is invariant under the group's action. We conclude that $\nu_B=c_B\cdot m_S$, where $m_S$ denotes the two dimensional Lebesgue measure restricted to $S$, and $c_B$ is some constant. For a crude bound on $c_B$ note that
$$
c_B=\frac{\nu_B(S)}{m_S(S)}=\frac{\mu(SB)}{m(S)}\le\frac{1}{m(S)}.
$$

\begin{lem}\label{lem:Nested}[The Nested Towers Lemma]
Let $(X,\mathcal B,\mu)$ be a standard probability space, and suppose $T:\C\rightarrow PPT(X)$ is a free $\C$-action. Let $\bset{a_n}_{n=1}^\infty$ be an increasing sequence of positive numbers such that $\sumit n 1 \infty \frac{a_n}{a_{n+1}}<\frac12$. Then there exists a sequence of sets $\bset{B_n}\subset \mathcal B$ such that
\begin{enumerate}[label=($N_\arabic*$),leftmargin=0.7cm]
\item $B_n$ is an $S_{a_n}$-set.
\item $S_{a_n}B_n\subset S_{a_{n+1}}B_{n+1}$.
\item $\mu\bb{S_{a_n}B_n}\nearrow1$ as $n\rightarrow\infty$.
\end{enumerate}
\end{lem}

This lemma, originally proven by Weiss in \cite{Weiss1996}, is a natural extension of Rokhlin's lemma about approximating actions of $\Z$ by Rokhlin towers.
Rokhlin's lemma for $\C$-action states that for every rectangle $R\subset\C$ and for every $\eps>0$ there exists a set $B$ which is an $R$-set such that $\mu(RB)>1-\eps$. This lemma is not enough as we would like the approximation of the space $X$ to be monotone.

The version of Rokhlin's lemma for $\C$-actions was proven by Lind in \cite{Lind1975}. We note that Weiss' definition for an $S$-set admits a weaker property than property $(F_2)$. Nevertheless, his proof of The Nested Tower Lemma extends to our definition of an $S$-set. For the reader's convenience the proof of this lemma can be found in the appendix.\\

Given a metric space $(Y,d)$ we let $\mathcal K(Y)$ denote the set of all compact subsets of $Y$. For every $A,B\in \mathcal K(Y)$ define the metric
$$
d_H(A,B):=\inf\bset{\eps>0, A\subset B^{+\eps}\text{ and } B\subset A^{+\eps}}.
$$
$d_H$ is called \underline{the Hausdorff distance induced by $(Y,d)$}.\\
We say that $(\mathcal K(Y),d_H)$ is \underline{the induced Hausdorff space of $(Y,d)$}.\\

Let $(X,\mathcal B,\mu)$ be a standard probability space, and let $\bset{B_n}$ be a sequence of sets given by the Nested Towers Lemma, Lemma \ref{lem:Nested}. Assume that $\mathcal P_{n-1}$ is a finite partition of $B_{n-1}$ into measurable sets $\bset{B_{n-1}^j}_{j=1}^{k_{n-1}}$. For every $x\in B_n$ and $1\le\ell\le k_{n-1}$ define the set 
$$
R_n^\ell(x):=\bset{z\in S_{a_n};\;T_zx\in B_{n-1}^\ell}.
$$
Since $B_{n-1}$ is an $S_{a_{n-1}}$-set, for every $z\neq w\in R_n^\ell(x)$ we have that $S_{a_{n-1}} T_zx\cap S_{a_{n-1}} T_wx=\emptyset$ which implies $\norm{z-w}\infty>2a_{n-1}$. In particular, $R_n^\ell(x)$ is a finite set and therefore compact.

Given $\delta>0$, we say a partition $\mathcal P_n=\bset{B_n^j}_{j=1}^{k_n}$ is a \underline{$\delta$-fine partition consistent with $\mathcal P_{n-1}$} if it is a finite measurable partition of $B_n$, and for every $1\le j\le k_n$ for every $x,y\in B_n^j$, $d_H(R_n^\ell(x),R_n^\ell(y))<\delta$ for every $1\le\ell\le k_{n-1}$.

\begin{lem}\label{lem:Partition}
Let $\bset{B_n}$ be a sequence of sets given by the Nested Towers Lemma. For every sequence of positive numbers $\bset{\delta_n}$ there exists a sequence of partitions $\bset{\mathcal P_n}$ such that for every $n$, $\mathcal P_n$ is a $\delta_n$-fine partition consistent with $\mathcal P_{n-1}$.
\end{lem}
\begin{proof} We will prove it by induction on $n$, where for the base step one could take any partition $\mathcal P_1$. Assume $\mathcal P_{n-1}$ was already defined and for every $1\le\ell\le k_{n-1}$ define the function $f_\ell:B_n\times B_n\rightarrow\R_+$ by
$$
f_\ell(x,y)=d_H(R_n^\ell(x),R_n^\ell(y)).
$$
For every $x\in B_n$ let
$$
D_n^\ell(x)=\bset{y\in B_n,\; f_\ell(x,y)<\frac{\delta_n}2}.
$$
As $f_\ell$ is a measurable function, these sets are measurable and form a cover for $B_n$. We will first show there exists a finite sub-cover of $\bset{D_n^\ell(x)}_{x\in B_n}$ for $B_n$. If no such sub-cover exists, then there exists a subsequence $\bset{x_m}$ such that for every $k\neq m$ we have $f_\ell(x_m,x_k)\ge\frac{\delta_n}2$ creating an infinite separated set in $\mathcal K(Y)$, which is a contradiction to the fact that the induced Hausdorff space of a totally bounded set is itself totally bounded (see Claim \ref{clm:inducedHousdorff}).

Let $Q_\ell$ denote the finite set of elements $x\in B_n$ forming the finite cover of $B_n$, and let $\mathcal Q_\ell$ denote the partition that we obtain by the collection $\bset{D_n^\ell(x)}_{x\in Q_\ell}$. Note that for every $A\in\mathcal Q_\ell$ there exists $\xi\in Q_\ell$ such that $A\subseteq D_n^\ell(\xi)$, and therefore for every $x,y\in A$, 
$$
f_\ell(x,y)\le f_\ell(x,\xi)+f_\ell(\xi,y)<\delta_n.
$$
We define $\mathcal P_n$ to be the refinement of all the partitions $\mathcal Q_\ell$, $\mathcal P_n=\bset{\bintersect\ell 1{k_{n-1}} A_\ell,\; A_\ell\in\mathcal Q_\ell}$. Let $A\in\mathcal P_n$, then for every $x,y\in A$ for every $1\le\ell\le k_{n-1}$ there exists $B\in\mathcal Q_\ell$ such that $x,y\in B$, and therefore $f_\ell(x,y)<\delta_n$, which implies that $\mathcal P_n$ is a $\delta_n$-fine partition consistent with $\mathcal P_{n-1}$, concluding our proof.
\end{proof}

\begin{claim}\label{clm:inducedHousdorff}
Let $(Y,d)$ be a metric space. If $Y$ is totally bounded, then the induced Hausdorff space, $(\mathcal K(Y),d_H)$, is totally bounded.
\end{claim}
For the reader's convenience we add a proof of this fact:
\begin{proof}
Let $\eps>0$. Since $Y$ is totally bounded there exists a finite sub-cover $\bset{B_j}_{j=1}^N$, such that for every $j$, $diam(B_j)<\eps$. For every $A\in \mathcal K(Y)$ we define the sequence:
$$
x_j(A)=	\begin{cases}
		1&, A\cap B_j\neq\emptyset\\
		0&, \text{otherwise}.
		\end{cases}
$$
Note that since the sequences' elements are only $\bset{0,1}$ and their length is finite, then the number of possible sequences is finite. Next, let $A$ and $B$ be sets such that they have the same sequence. Fix $a\in A$, then there exists $j$ such that $a\in B_j$ as $\bset{B_j}$ is a cover for $Y$. Since the sequence of $A$ and $B$ are the same $1=x_j(A)=x_j(B)$ and so there exists $b\in B\cap B_j$ and in particular as $diam(B_j)<\eps$, we get that $d(a,b)<\eps$. This shows that $A\subset B^{+\eps}$. By a symmetric argument, $B\subset A^{+\eps}$ as well. We conclude that $d_H(A,B)\le\eps$.\\
For every $1\le j\le N$ we arbitrarily chose an element $b_j\in B_j$, and define for every sequence $\bset{x_j}$ the set $D_{\bset{x_j}}:=\bset{b_j; x_j=1}$. The collection $\bset{D_{\bset{x_j}}}$ forms a finite $\eps$-net for the space $(\mathcal K(Y),d_H)$ concluding the proof.
\end{proof}
\section{Construction of a special sequence}
In this section we will use all the lemmas proven in the previous sections to construct a special sequence of functions that will be used in the proof of Theorem \ref{thm:UprLwrBnd}. Beyond this section the only thing one needs to keep in mind is the following lemma:
\begin{lem}\label{lem:SuperSequence}
Let $(X,\mathcal B,\mu)$ be a standard probability space, and suppose $T:\C\rightarrow PPT(X)$ is a free $\C$-action. Let $\bset{a_n}$ be the sequence defined by $a_1=1$, and $a_{n}=  D\cdot n\log^2n\cdot a_{n-1}$ for every $n\ge2$ where $  D>0$ is some parameter sufficiently large. There exists a sequence of measurable sets $\bset{X_n},\; X_n\nearrow X$, and a sequence of measurable functions $F_n:X_n\rightarrow\C$ with the following properties:
\begin{enumerate}[label=(\roman*),leftmargin=0.5cm]
\item {
$$
\mu\bb{\bset{x\in X_1,\;\abs {F_1(x)}\le\frac14}}\ge\frac1{25}\;\;\text{ and }\;\; \mu\bb{\bset{x\in X_1,\;\abs {F_1(x)}\ge\frac34}}\ge\frac1{25}.
$$}
\item There exists a sequence of measurable sets $\bset{G_n}$ such that $G_n\subset X_n$ and:
\begin{enumerate}[label=(\Alph*),leftmargin=0.7cm]
\item
$$
\mu\bb{X_n\setminus G_n}\le\mu\bb{X_n\setminus X_{n-1}}+\frac1D\cdot O\bb{\frac1{n\log^2n}}.
$$
\item For every $x\in G_n$:
\begin{enumerate}[label=($B_{\arabic*}$),leftmargin=0.9cm]
\item $S_{a_{n-2}}x\subseteq X_{n-1}$, implying that $S_{a_{n-2}}G_n\subseteq X_{n-1}\subseteq X_n$.
\item The function $F_n^x: {\C}\rightarrow\C$ defined by $F_n^x(z)=F_n(T_zx)$ is holomorphic in $S_{a_{n-2}}$.
\item
$
\underset{z\in S_{a_{n-2}}}\max\;\abs{F_n(T_zx)-F_{n-1}(T_zx)}<10^{-2n}.
$
\item For every $1\le m\le n-2$:
$$
\underset{z\in S_{a_m}}\max\;\abs{F_n(T_zx)}\le2\exp\bb{2^{1-B}\mathcal M_B(m+1)},
$$
where $B$ is a numerical constant sufficiently large, and
$$
\mathcal M_B(m)=\exp\bb{B\cdot m+\pi D^2\sumit j 2 {m-1} j^2\log^4j}.
$$
\end{enumerate}
\end{enumerate}
\end{enumerate}
\end{lem}
\begin{proof}
Let $\bset{B_n}$ be a sequence of sets obtained for the sequence $\bset{a_n}$ by Weiss' Nested Towers Lemma, Lemma \ref{lem:Nested}, such that $\mu\bb{X\setminus S_{a_1}B_1}<\frac 1{200}$.

We will set $X_n:=S_{a_n}B_n$, and define the sequence of functions $F_n$  as a linear combination of step functions,
$$
F_n(T_zx)=F_n(z,x)=\sumit j1{k_n}F_n^j(z)\cdot\indic{B_n^j}(x),
$$ 
where $\bset{F_n^j}_{j=1}^{k_n}$   {are entire}, and $\bset{B_n^j}_{j=1}^{k_{n-1}}$ is a measurable partition of $B_n$, denoted $\mathcal P_n$. Note that $F_n$ is well defined, since $B_n$ is an $S_{a_n}$-set and therefore the mapping $T_zx\mapsto (z,x)$ is well defined, as mentioned in Section 2.2.

Formally, we will construct this sequence inductively. Define $F_1:X_1\rightarrow\C$ by
$$
F_1(T_zx)=F_1(x,z)=z.
$$
$F_1$ is measurable, since it is constant with respect to one variable, and continuous with respect to the other. By the way $F_1$ is defined for every $x\in B_1$ and $\abs z<\frac14$
$$
\abs{F_1(T_zx)}=\abs{z}<\frac14,
$$
and so $\bset{F_1\le\frac14}\supset\frac14\D B_1$. Following Remark \ref{rmk:JonsObservation}:
 {\begin{eqnarray*}
\mu\bb{\bset{x\in X_1,\;\abs {F_1(x)}\le\frac14}}&\ge&\mu\bb{\frac14\D B_1}=m\bb{\frac14\D}\cdot\frac{\mu\bb {S_{a_1}B_1}}{m\bb{S_{a_1}}}\\
&\overset{\text{as }\atop a_1=1}=&\frac{\pi}{4^3}\cdot\mu\bb{S_{a_1}B_1}>\frac{\pi}{4^3}\cdot\frac{199}{200}>\frac1{25}.
\end{eqnarray*}}
A similar computation shows that $\mu\bb{\bset{\abs {F_1}\ge\frac34}}>\frac1{25}$ as well, and so property (i) holds.\\

Assume that $F_{n-1}: X_{n-1}\rightarrow\C$ was defined as
$$
F_{n-1}(T_zx)=F_{n-1}(x,z)=\sumit j 1 {k_{n-1}}F_{n-1}^j(z)\indic{B_{n-1}^j}(x),
$$
and that property (ii) holds for $F_{n-1}$. We assume in addition that instead of property $(B_4)$ we have property $(B_4')$: for the same parameter $B$
$$
\underset{z\in S_{a_m}}\max\;\abs{F_n(T_zx)}\le\exp\bb{2^{1-B}\mathcal M_B(m+1)}+\sumit j 1 n 10^{-2j}.
$$
Naturally, property $(B_4')$ implies property $(B_4)$. Moreover, we assume that for every $1\le j\le k_{n-1}$
$$
\underset{z\in S_{a_{n-1}}}\max\abs{F_{n-1}^j(z)}\le \exp\bb{2^{1-B}\mathcal M_B(n)}.
$$
We refer to this property as property $(B_5)$, and regard it as part of property $(ii')$, which is property $(ii)$ where $(B_4)$ is replaced by $(B_4')$ and $(B_5)$ is added.\\
Since $F_{n-1}^j$ is entire for every $j$ fixed, it is uniformly continuous on $S_{a_{n-1}}^{+1}$, and therefore there exists $\delta_n\in(0,1)$ such that for every $1\le j\le k_{n-1}$:
$$
\underset{z,w\in S_{a_{n-1}}^{+1}\atop\abs{z-w}<\delta_n}\sup\; \abs{F_{n-1}^j(z)-F_{n-1}^j(w)}<\frac{10^{-2n}}2.
$$
Let $\mathcal P_n$ be a partition of $B_n$ which is $\delta_n$-fine and consistent with $\mathcal P_{n-1}=\bset{B_{n-1}^\ell}_{\ell=1}^{k_{n-1}}$, the partition of $B_{n-1}$ used to define $F_{n-1}$. Such a partition exists by Lemma \ref{lem:Partition}. \\
For every $j$ we use the axiom of choice to choose a representative $x_n^j\in B_n^j$. We will define the function $F_n^j$ by using Lemma \ref{lem:EntireExtension} with the parameters:
\begin{eqnarray*}
&&\Lambda_n^j=\bset{\frac\lambda{a_{n-1}},\; \lambda\in S_{a_n}\text{ so that } T_\lambda x_n^j\in B_{n-1}} ,\; C=\frac{a_n}{a_{n-1}}=  D\cdot n\log^2n,\\
&&\;\mathcal M:=\exp\bb{2^{1-B}\mathcal M_B(n)},\; f_\lambda(z):=F_{n-1}(T_{a_{n-1} \bb{\lambda+z}}\; x_n^j):S_1\rightarrow \C.
\end{eqnarray*}
Let us verify these parameters satisfy the requirements of the lemma. First of all, because $B_{n-1}$ is an $S_{a_{n-1}}$ set, then for every $\lambda\neq\mu\in\Lambda_n^j$, we have that 
$$
S_{a_{n-1}}(a_{n-1}\cdot \lambda)\cap S_{a_{n-1}}(a_{n-1}\cdot \mu)=\emptyset\iff S_1(\lambda)\cap S_1(\mu)=\emptyset.
$$
In particular, for every $\lambda\neq\mu\in\Lambda_n^j$, we have $\norm{\lambda-\mu}{\infty}>2$.\\
Next, for every $n\ge 1$ for every $D$ large enough
\begin{eqnarray*}
\mathcal M&=&\exp\bb{2^{1-B}\mathcal M_B(n)}=\exp\bb{2^{1-B}\exp\bb{B\cdot n+  {D^2}\pi\sumit j 2 {n-1}j^2\log^4j}}\\
&\ge&\exp\bb{2^{1-B}\exp\bb{B \cdot n+  D^2\pi\cdot n\log^2n}}.
\end{eqnarray*}
In fact a more accurate lower bound is
$$
\exp\bb{2^{1-B}\exp\bb{B \cdot n+  D^2\pi\cdot n^2\log^4n}},
$$
but the bound indicated above is enough for our use. 
In particular, for every constant $B$ there exists $D$ large enough so that
$$
\mathcal M\ge 40\log\bb{D\cdot n\log^2n}.
$$
We conclude that all the requirements of Lemma \ref{lem:EntireExtension} are satisfied.

Define the function
\begin{eqnarray}\label{eq:defOfFn}
F_n(T_zx)=F_n(z,x)=\sumit j 1 {k_n}F_n^j(z)\cdot\indic {B_n^j}(x).
\end{eqnarray}
Note that every summand in the sum is a measurable function as it is an indicator function of the measurable set $B_n^j$ in one variable and the continuous function $F_n^j$ in the other. This implies that $F_n$ is a measurable function since the number of sets in the partition (and therefore the number of summands in the sum) is finite.\\
To conclude the proof it is left to show that property (ii') holds for $F_n$ as well.\\
Fix $1\le j\le n_{k-1}$ and $\lambda\in \Lambda_n^j$, and recall the definition of the sets $A_\lambda$ in Lemma \ref{lem:EntireExtension}:
$$
A_\lambda=S_1(\lambda)\cap\bset{z,\; \abs{f(z)-f_\lambda(z-\lambda)}<\exp\bb{-\frac{\mathcal M}4}},
$$
where for us $f=F_n^j$ for some $1\le j\le n_{k-1}$. Let $D_\lambda=a_{n-1}\cdot A_\lambda$, and define the set
$$
G_n:=\bunion j 1 {k_{n-1}}\bb{\underset{\lambda\in\Lambda_n^j}\bigcup D_\lambda^{-1-a_{n-2}}}B_n^j\subseteq X_n.
$$
We will first show that property (A) holds. By the way the partition $\mathcal P_n$ was defined, 
\begin{eqnarray}\label{eq:inclusion}
\bunion j 1 {k_n}\bb{\underset{\lambda\in\Lambda_n^j}\bigcup  S_{a_{n-1}}^{-\delta_n}(a_{n-1}\cdot\lambda)}B_n^j\subseteq X_{n-1}\subseteq\bunion j 1 {k_n}\bb{\underset{\lambda\in\Lambda_n^j}\bigcup  S_{a_{n-1}}^{+\delta_n}(a_{n-1}\cdot\lambda)}B_n^j
\end{eqnarray}
(see Figure \ref{label:distortion}).
\begin{figure}[!ht]
\centering
  \begin{center}
    \includegraphics[width=0.8\textwidth]{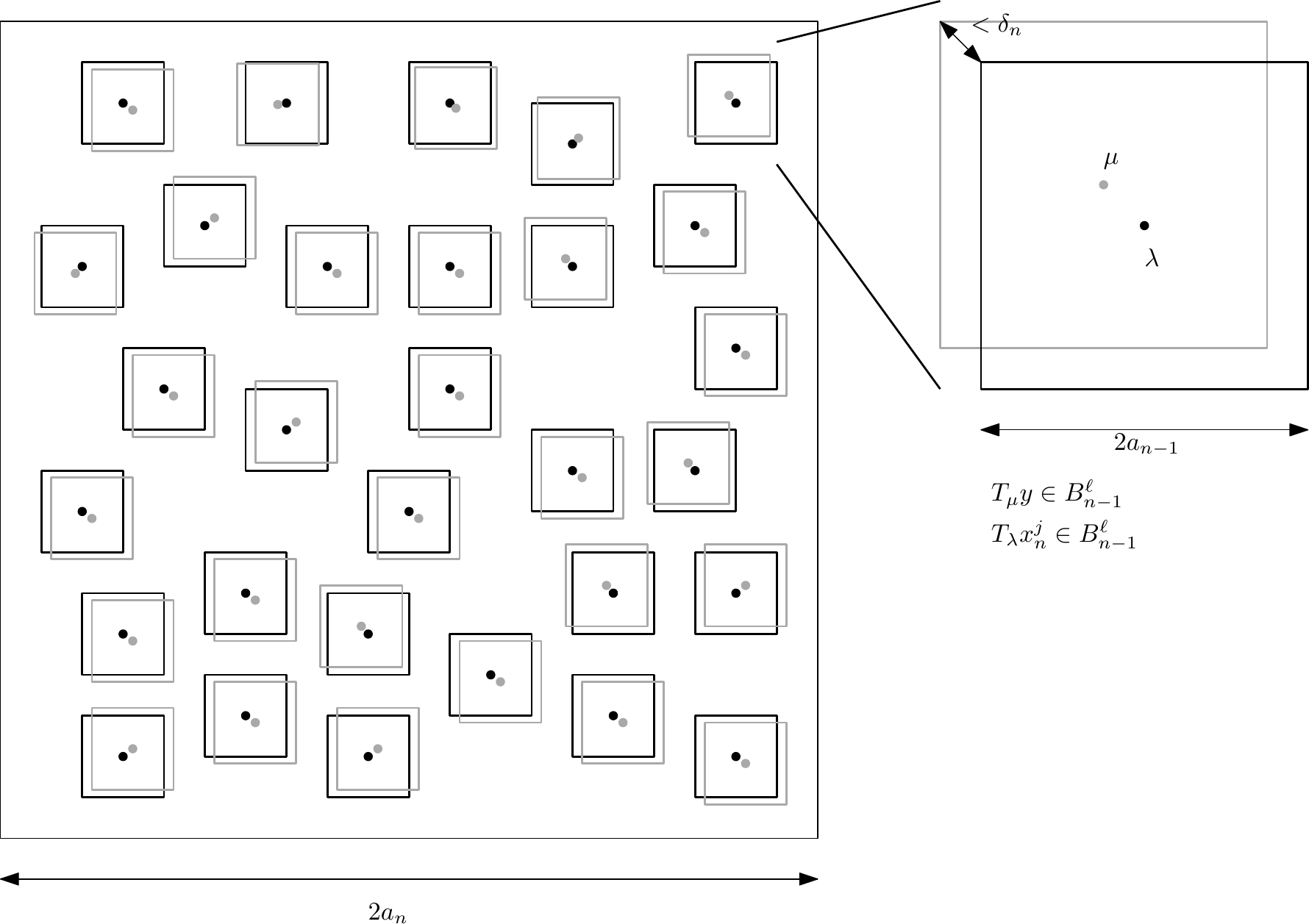}
  \end{center}
  \caption{ Fix $n$ and $1\le j\le k_n$. The black points represent the set $\Lambda_n^j$. For every $y\in B_n^j$, $y\neq x_n^j$, the gray configuration of squares represent the case where $y$ was chosen to be the representative of $B_n^j$ instead of $x_n^j$. Thus, the gray configuration of squares is a distortion of the black configuration of squares by at most $\delta_n$, by the way the partition $\mathcal P_n$ was defined.}
  \label{label:distortion}
  \end{figure}

We have
\begin{eqnarray*}
\mu\bb{X_n\setminus G_n}&\le&\mu\bb{X_n\setminus\bunion j 1 {k_n}\bb{\underset{\lambda\in\Lambda_n^j}\bigcup  S_{a_{n-1}}(a_{n-1}\cdot\lambda)}B_n^j}\\
&+&\mu\bb{\bunion j 1 {k_n}\bb{\underset{\lambda\in\Lambda_n^j}\bigcup \bb{S_{a_{n-1}}(a_{n-1}\cdot\lambda)\setminus D_\lambda^{-1-a_{n-2}}}}B_n^j}.\\
\end{eqnarray*}
We will use Remark \ref{rmk:JonsObservation} to bound each of these terms:\\
Remember that following Lemma \ref{lem:EntireExtension}, $m\bb{S_1(\lambda)\setminus A_\lambda^{-\eps}}=O\bb{\frac1C+\eps}$. We obtain that for every $\lambda\in\Lambda_n^j$, 
$$
m\bb{S_{a_{n-1}}\setminus \bb{a_{n-1}\cdot A_\lambda}^{-1-a_{n-2}}}\le O(1)\cdot a_{n-1}^2\bb{\frac{a_{n-1}}{a_n}+\frac{1+a_{n-2}}{a_{n-1}}}.
$$
Define $S_j:=\bb{\underset{\lambda\in\Lambda_n^j}\bigcup \bb{S_{a_{n-1}}(a_{n-1}\cdot\lambda)\setminus D_\lambda^{-1-a_{n-2}}}}$. Then since $B_n^j$ is an $S_{a_n}$-set,
\begin{eqnarray*}
\mu\bb{\bunion j 1 {k_n}S_jB_n^j}&\overset{Rmk.\atop \ref{rmk:JonsObservation}}=&\sumit  j 1 {k_n}m\bb{S_j}\cdot\frac{\mu\bb{S_{a_n}B_n^j}}{m(S_{a_n})}\\
&\le&\sumit  j 1 {k_n}\underset{\lambda\in\Lambda_n^j}\sum\; m\bb{S_{a_{n-1}}(a_{n-1}\cdot\lambda)\setminus \bb{a_{n-1}\cdot A_\lambda}^{-1-a_{n-2}}}\cdot\frac{\mu\bb{S_{a_n}B_n^j}}{m(S_{a_n})}\\
&\le&\sumit  j 1 {k_n}\#\Lambda_n^\cdot O(1)\cdot \frac{a_{n-1}^2\bb{\frac{a_{n-1}}{a_n}+\frac{1+a_{n-2}}{a_{n-1}}}}{4a_n^2}\cdot\frac{\mu\bb{S_{a_n}B_n^j}}{m(S_{a_n})}\\
&\le&\frac{a_n^2}{a_{n-1}^2}\cdot O(1)\cdot \frac{a_{n-1}^2\bb{\frac{a_{n-1}}{a_n}+\frac{1+a_{n-2}}{a_{n-1}}}}{4a_n^2}\sumit  j 1 {k_n}\mu\bb{S_{a_n}B_n^j}\le O\bb{\frac{a_{n-2}}{a_{n-1}}}.
\end{eqnarray*}
Similarly,
\begin{eqnarray*}
&&\mu\bb{X_n\setminus\bunion j 1 {k_n}\bb{\underset{\lambda\in\Lambda_n^j}\bigcup  S_{a_{n-1}}(a_{n-1}\cdot \lambda)}B_n^j}\\
&\le& \mu\bb{X_n\setminus\bunion j 1 {k_n}\bb{\underset{\lambda\in\Lambda_n^j}\bigcup  S_{a_{n-1}}^{+1}(a_{n-1}\cdot\lambda)}B_n^j}+\mu\bb{\bunion j 1 {k_n}\bb{\underset{\lambda\in\Lambda_n^j}\bigcup \bb{S_{a_{n-1}}^{+1}\setminus S_{a_{n-1}}}(a_{n-1}\cdot \lambda)}B_n^j}\\
&\overset{(\ref{eq:inclusion})}\le&\mu\bb{X_n\setminus X_{n-1}}+\sumit j 1 {k_n}\#\Lambda_n^j \cdot \frac{2\cdot a_{n-1}}{a_n^2}\cdot\mu(S_{a_n}B_n^j)<\mu\bb{X_n\setminus X_{n-1}}+\frac{2}{a_{n-1}}.
\end{eqnarray*}
Overall, we get that:
$$
\mu\bb{X_n\setminus G_n}\le \mu\bb{X_n\setminus X_{n-1}}+O\bb{\frac{a_{n-2}}{a_{n-1}}}=\mu\bb{X_n\setminus X_{n-1}}+\frac1D\cdot O\bb{\frac1{n\log^2n}},
$$
concluding the proof.

Next, we will show that for every $x\in G_n$ the properties enumerated as $(B)$ hold. Fix $x\in G_n$. There exists $1\le j\le k_n$, and $\lambda\in \Lambda_n^j$, such that $x\in D_\lambda^{-1-a_{n-2}}B_n^j$.~\\

\underline{Property $(B_1)$ holds:} Note that $T_\lambda x_n^j\in B_{n-1}$ by the way $\Lambda_n^j$ was defined. Similarly, for every $y\in B_n^j$ there exists $\mu\in\C$ such that $\abs{\mu-\lambda}<\frac{\delta_n}{a_{n-1}}<\frac1{a_{n-1}}$ and $T_{a_{n-1}\cdot\mu} y\in B_{n-1}$. Let $y\in B_n^j$ be so that $x=T_wy$ for $w\in D_\lambda^{-1-a_{n-2}}$. Then for every $z\in S_{a_{n-2}}$ we have that $T_zx=T_{z+w}y\in D_\lambda^{-1}y\subset D_\mu^{-1+\delta_n}y\subset D_\mu y\subset X_{n-1}$, as needed.\\

\underline{Property $(B_2)$ holds:} Since $F_n^j$ is  {entire}, we get that $z\mapsto F_n(T_zx)$ is  {entire for every $x\in G_n$}.\\

\underline{Property $(B_3)$ holds:} We want to show that for every $z\in S_{a_{n-2}}$,
$$
\abs{F_n(T_zx)-F_{n-1}(T_zx)}<10^{-2n}.
$$
The idea is that $\delta_n$ was chosen so that for every $j$ if the function $F_{n-1}^j$ is perturbed by something smaller than $\delta_n$, then its image is perturbed by something which is bounded by $\frac{10^{-2n}}2$. Next, we take a partition which is $\delta_n$-fine partition consistent with $\mathcal P_{n-1}$, which means that for every $y\in B_n^j$, the configuration of squares associated with it is at most a $\delta_n$-distortion of the configuration of squares associated with $x_n^j$, meaning $\bset{f_\lambda}$ used to construct $F_n$ differ by at most $\frac{10^{-2n}}2$ from the ones used if $y$ was the chosen representative. Combining this with the fact $F_n$ approximates these $f_\lambda$ to begin with, we get that $\abs{F_n-F_{n-1}}$ is small.

Formally, let $x_0\in B_{n-1}^\ell\cap S_{a_n}B_n^j$, then there exist $\lambda\in\Lambda_n^j$ and $w\in S_{a_n}$ such that $x_0=T_wy,\; y\in B_n^j$, and $\abs{\lambda-w}<\delta_n$. By the way $F_n$ is constructed, for every $z\in D_\lambda^{-\delta_n}$:
\begin{eqnarray*}
\abs{F_{n-1}(T_zx_0)-F_n(T_zx_0)}&=&\abs{F_{n-1}^\ell(z)-F_n(T_{z+w}T_{-w}x_0)}=\abs{F_{n-1}^\ell(z)-F_n(T_{z+w}y)}\nonumber\\
&\overset{\text{by }(\ref{eq:defOfFn})}=&\abs{F_{n-1}^\ell(z)-F_n^j(z+w)}\\
&\le&\abs{F_{n-1}^\ell(z)-F_{n-1}^\ell(z+w-\lambda)}+\abs{F_{n-1}^\ell(z+w-\lambda)-F_n^j(z+w)}.\nonumber
\end{eqnarray*}
Now, by property $(E_1)$ of $F_n^j$, guaranteed by Lemma \ref{lem:EntireExtension}, we know that since $\abs{\lambda-w}<\delta_n$, then $z+w\in D_\lambda$ and so-
$$
\abs{F_{n-1}^\ell(z+w-\lambda)-F_n^j(z+w)}<\exp\bb{-\frac{\mathcal M_B(n)}4}<\frac{10^{-2n}}2.
$$
On the other hand, since $\abs{\lambda-w}<\delta_n$,
\begin{eqnarray*}
\abs{F_{n-1}^\ell(z)-F_{n-1}^\ell(z+w-\lambda)}\le\underset{\zeta,\xi\in S_{a_{n-1}}^{+1}\atop\abs{\zeta-\xi}<\delta_n}\sup\;\abs{F_{n-1}^\ell(\zeta)-F_{n-1}^\ell(\xi)}<\frac{10^{-2n}}2
\end{eqnarray*}
as well. Overall, for every $x_0\in B_{n-1}\cap S_{a_n}B_n^j$ we have that
\begin{eqnarray}\label{Eq:consistency}
\abs{F_n(T_zx_0)-F_{n-1}(T_zx_0)}<10^{-2n}.
\end{eqnarray}
Next, since $x\in D_\lambda^{-1-a_{n-2}}B_n^j$ there exists $w\in D_\lambda^{-1-a_{n-2}}$ such that $T_{-w}x\in B_n^j$. In addition, since the partition $\mathcal P_n$ is $\delta_n$-fine there exists $\zeta$ such that $\abs{\zeta-\lambda}<\delta_n$, and $T_{-w-\zeta}x\in B_{n-1}^\ell$. For every $z\in S_{a_{n-2}}$ 
$$
z+w\in D_\lambda^{-1}\Rightarrow z+w-\zeta+\lambda\in D_\lambda^{-1+\delta_n}\subset D_\lambda^{-\delta_n},
$$
and by using (\ref{Eq:consistency}) we get that if $x_0=T_{-w-\zeta}x\in B_{n-1}^\ell\cap S_{a_n}B_n^j$, then
\begin{eqnarray*}
\abs{F_n(T_zx)-F_{n-1}(T_zx)}&=&\abs{F_n(T_{z+w+\zeta}T_{-w-\zeta}x)-F_{n-1}(T_{z+w+\zeta}T_{-w-\zeta}x)}\\
&=&\abs{F_n(T_{z+w+\zeta}x_0)-F_{n-1}(T_{z+w+\zeta}x_0)}<10^{-2n},
\end{eqnarray*}
and property $(B_3)$ holds.\\

\underline{Property $(B_4')$ holds:} Note that for every $m\le n-3$ we have that by property $(B_3)$, and the induction assumption (which holds only for $m\le n-3$),
$$
\underset{z\in S_{a_{m}}}\max\;\abs{F_n(T_zx)}\le\underset{z\in S_{a_{m}}}\max\;\abs {F_{n-1}(T_zx)}+10^{-2n}\le \exp\bb{2^{1-B}\mathcal M_B(m+1)}+\sumit j 1 n 10^{-2j}.
$$
For $m=n-2$, by property $(B_3)$,
$$
\underset{z\in S_{a_{n-2}}}\max\;\abs{F_n(T_zx)}\le\underset{z\in S_{a_{n-1}}}\max\;\abs{F_{n-1}(T_zx)}+10^{-2n}\le \exp\bb{2^{1-B}\mathcal M_B(n)}+10^{-2n}.
$$

\underline{Property $(B_5)$ holds:} By property $(E_2)$ of Lemma \ref{lem:EntireExtension} for every $1\le j\le k_n$:
\begin{eqnarray*}
\underset{z\in S_{a_n}}\max\abs{F_n^j(z)}&\le& \exp\bb{2^{1-B}\underset{1\le\ell\le k_{n-1}}\max\underset{z\in S_{a_{n-1}}}\max\abs{F_{n-1}^\ell(z)}\cdot \exp\bb{\pi C^2}}\\
&\le&\exp\bb{2^{1-B}\mathcal M_B(n)\exp\bb{\pi\cdot n^2\log^4n}}\le\exp\bb{2^{1-B}\mathcal M_B(n+1)}.
\end{eqnarray*}

This concludes the proof of the lemma.
\end{proof}
\section{The Proof of Theorem \ref{thm:UprLwrBnd}}
Let $\bset{F_n}$ be the sequence constructed in Lemma \ref{lem:SuperSequence}.

\subsection{The sequence $\bset{F_n}$ converges almost surely to a measurably entire function:}

Let $x\in\bunion m 1 \infty\bintersect k m \infty G_k$, then by property $(B_3)$, $\bset{F_n(T_zx)}$ converges locally uniformly to an entire function, and in particular if $\bset {F_n}$ converges almost surely, then it converges to a measurably entire function. It is therefore enough to show that $\mu\bb{\bunion m 1 \infty\bintersect k m \infty G_k}=1$ to conclude the proof. To see that $\mu\bb{\bunion m 1 \infty\bintersect  k m\infty G_k}=1$ we will show that 
$$
0=\mu\bb{\bb{\bunion m 1 \infty\bintersect  km \infty G_k}^c}=\mu\bb{\bintersect m 1 \infty\bunion  km \infty G_k^c}=\limit m\infty \mu\bb{\bunion  km \infty G_k^c}.
$$
By using property $(A)$ of the sequence $\bset{F_n}$ and the fact that $\bset{X_n}$ is increasing, we obtain that
\begin{eqnarray*}
\mu\bb{\bunion  km \infty G_k^c}&\overset{G_k\subseteq X_k}=&\mu\bb{\bunion km \infty \bb{X\setminus X_k}\uplus\bb{X_k\setminus G_k}}\le\mu\bb{\bunion km \infty \bb{X\setminus X_k}}+\sumit k m \infty\mu\bb{X_k\setminus G_k}\\
&\overset{\text{Property (A)}}\le&\mu\bb{X\setminus X_m}+\sumit  km \infty\bb{\mu\bb{X_k\setminus X_{k-1}}+\frac1D\cdot O\bb{\frac1{k\log^2k}}}\\
&\overset{X_{n-1}\subseteq X_n}\le&2\mu\bb{X\setminus X_{m-1}}+\frac{O(1)}D\sumit k m \infty \frac1{k\log^2k}.
\end{eqnarray*}
To conclude the proof, note that the latter tends to zero as $m$ tends to $\infty$, since the series converges.

\subsection{The limiting function $F$ is not constant:} 
Since the sequence $\bset{F_n}$ converges in measure to a function which we shall denote by $F$,
\begin{eqnarray*}
\mu\bb{\bset{\abs F\le\frac13}}=\limit n\infty\mu\bb{\bset{\abs {F_n}\le\frac13}}\;\;,\;\;\mu\bb{\bset{\abs F\ge\frac23}}=\limit n\infty\mu\bb{\bset{\abs {F_n}\ge\frac23}}.
\end{eqnarray*}
We will bound each of the quantities above from bellow by a uniform constant for every n.\\
\begin{eqnarray*}
\mu\bb{\bset{\abs {F_n}\le\frac13}}&\overset{\text{property}\atop(B_3)}\ge&\mu\bb{\bset{\abs {F_{n-1}}\le\frac13-10^{-2n}}\setminus G_n^c}\\
&\overset{\text{property}\atop(B_3)}\ge&\cdots\overset{\text{property}\atop(B_3)}\ge\mu\bb{\bset{\abs {F_1}\le\frac13-\sumit m 1 n10^{-2m}}\setminus\bunion m 1 n G_m^c}\\
&\ge&\mu\bb{\bset{\abs {F_1}\le\frac14}}-\mu\bb{\bunion m 1 n G_m^c}\\
&\overset{(A)}\ge&\frac1{25}-2\mu\bb{X\setminus X_1}-\frac{O(1)}{ D}\sumit n 2 \infty \frac1{n\log^2n}=\frac3{100}-\frac{O(1)}{D}\sumit n 2 \infty \frac1{n\log^2n},
\end{eqnarray*}
where $  D$ is the constant from the definition of the sequence $\bset{a_n}$. If we choose $  D$ large enough we get
$$
\mu\bb{\bset{\abs {F_n}\le\frac13}}\ge\frac1{100}.
$$
A similar computation shows that $\mu\bb{\bset{\abs F\ge\frac23}}$ is greater than the same constant, concluding that $F$ is not constant.

\subsection{Upper bound for the growth rate of the function:}

Let $x\in \bintersect kn\infty G_k$, we will show that (\ref{eq:UprBnd}) holds. For every $k\ge n$ by property $(B_4)$ of the sequence $\bset{F_n}$:
\begin{eqnarray*}
\underset{z\in S_{a_m}}\max\;\abs{F(T_zx)}&\le& 2\exp\bb{2^{1-B}\mathcal M_B(m+1)}.
\end{eqnarray*}
In addition, as $\frac{a_{m+1}}{a_m}\sim m\log^2 m$, for every $\eps>0$ for every $m>m_\eps$ large enough
\begin{eqnarray*}
\mathcal M_B(m)&=&\exp\bb{B\cdot m+\pi\sumit k 2{m}\bb{\frac{a_{k}}{a_{k-1}}}^2}\le\exp\bb{B\cdot m+O(1)\cdot D\sumit k 2{m}k^2\log^4k}\\
&\le&\exp\bb{B\cdot m+O(1)\cdot D m^3\log^4m}\le\exp\bb{O(1)\cdot\log^{3+\frac\eps2} a_m}.
\end{eqnarray*}
We conclude that for every $\eps>0$
$$
\frac{\log\log\underset{z\in S_{a_m}}\max\;\abs {F(T_zx)}}{\log^{3+\eps} a_m}\le\frac{O(1)}{\log^{\frac\eps2}a_m}\underset{m\rightarrow\infty}\longrightarrow0.
$$
For every $R$ large enough, let $m$ be such that $a_m\le R<a_{m+1}$. Using the estimate above with $m+1$ and $\frac\eps2$ instead of $\eps$ we get,
$$
\frac{\log\log\underset{z\in S_R}\max\;\abs {F(T_zx)}}{\log^{3+\eps} R}\le\frac{\log\log\underset{z\in S_{a_{m+1}}}\max\;\abs {F(T_zx)}}{\log^{3+\eps} a_m}\le \frac{O(1)}{\log^{\frac\eps2}a_m}\cdot\bb{\frac{\log\bb{a_{m+1}}}{\log \bb{a_m}}}^{3+\frac\eps2}\underset{m\rightarrow\infty}\longrightarrow0.
$$
concluding the proof of the theorem.
\section{Appendix}
 For completeness we introduce here a proof of the Nested Towers Lemma:\\
\textbf{ Lemma \ref{lem:Nested}. }{\em [The Nested Towers Lemma]
Let $(X,\mathcal B,\mu)$ be a standard probability space, and suppose $T:\C\rightarrow PPT(X)$ is a free $\C$-action. Let $\bset{a_n}_{n=1}^\infty$ be an increasing sequence of positive numbers such that $\sumit n 1 \infty \frac{a_n}{a_{n+1}}<\frac12$. Then there exists a sequence of sets $\bset{B_n}\subset \mathcal B$ such that}
\begin{enumerate}[label=($N_\arabic*$),leftmargin=0.7cm]
\item $B_n$ is an $S_{a_n}$-set.
\item $S_{a_n}B_n\subset S_{a_{n+1}}B_{n+1}$.
\item $\mu\bb{S_{a_n}B_n}\nearrow1$ as $n\rightarrow\infty$.
\end{enumerate}
We open this appendix with a discussion of a preliminary lemma. This lemma is a version of Rokhlin's lemma for flows. It was proven by Lind in \cite{Lind1975}: 

\begin{lem}\label{lem:Lind}
Let $T$ be a free $n$-dimensional flow on a standard probability space $(X,\mathcal B,\mu)$. Then for any rectangle $Q\subset\R^n$ and $\eps> 0$, there exists a $Q$-set, $F\subset X$ such that $\mu\bb{T_QF} > 1 - \eps$.
\end{lem}

  {\begin{rmk}
Note that unlike our definition of an $S$-set, Lind's definition does not require measurability. Namely, his definition lacks condition $(F_2)$ completely. Nevertheless, he proved that the set $F$ found in Lemma \ref{lem:Lind}, not only satisfies condition $(F_2)$, but in fact fulfills a stronger condition than the one we impose. For more information see ~\cite[p. 177]{Lind1975}.
\end{rmk}}

We first describe the idea of the proof. We start with a sequence of sets $\bset{B_n(0)}$ obtained by Rokhlin's lemma for a sequence $\bset{\eps_n}$, and define for every $n$ a sequence of sets $\bset{B_n(k)}_{k=0}^\infty$: $B_n(k+1)$ will only include elements of $B_n(k)$ so that their restricted orbit, $S_{a_n}x$, is included in $S_{a_{n+1}}B_{n+1}(k)$. Then we will bound the measure of the sets that we remove to conclude that for $B_n:=\bintersect k 0\infty B_n(k)$ we have $S_{a_n}B_n\nearrow X$ as $n\rightarrow\infty$.
\begin{proof}
Let $\bset{\eps_n}$ be a positive monotone decreasing sequence such that $\sumit n 1 \infty\eps_n<\infty$. By Rokhlin's Lemma for flows, Lemma \ref{lem:Lind},  for every $n$ there exists a set $B_n(0)$ such that
\begin{enumerate}[label=($L_\arabic*$),leftmargin=1cm]
\item $B_n(0)$ is an $S_{a_n}$-set.
\item $\mu\bb{S_{a_n}B_n(0)}>1-\eps_n$.
\end{enumerate}
We inductively define the sets:
$$
B_j(k+1)=\bset{x\in B_j(k),\; S_{a_j}x\subset S_{a_{j+1}}B_{j+1}(k)}.
$$
First of all, $B_j(k+1)\subseteq B_j(k)$ and so the set $B_j:=\bintersect k 1 \infty B_j(k)$ is well defined. Next, every measurable subset of $B_j(0)$, is in itself an $S_{a_j}$-set, because of property $(F_2)$ of an $S_{a_j}$-set. If for every $k$, $B_j(k)$ is measurable, then it is an $S_{a_j}$-set, and so is $B_j$. To conclude that property $(N_1)$ holds it is left to show that for every $k$ the set $B_j(k)$ is measurable.

It is clear that the inclusion condition, condition $(N_2)$, holds by the way the sequence $\bset{B_j}$ is defined. To prove that $(N_1),\;(N_3)$ hold we will need the following claim:

\underline{Claim:} Let $x\in B_j(k)$, then $x\in B_j(k+1)$ if and only if there exists $y\in B_{j+1}(k)$ such that $S_{a_j}x\subset S_{a_{j+1}}y$.\\

This claim tells us that for every $x$ that we threw away on step $k$ of the construction of the sequence $\bset{B_j(k)}$, $x\in B_j(k)\setminus B_j(k+1)$, its restricted orbit, $S_{a_j}x$, is included in the restricted orbit of some element $y\in B_{j+1}(k-1)$ that we threw away on step $(k-1)$ of the construction of the sequence $\bset{B_{j+1}(k)}$, $y\in B_{j+1}(k-1)\setminus B_{j+1}(k)$.

\underline{Proof of the claim:} The `if' part of the claim is obvious. To prove the other side, assume by contrudiction that the set defined by 
$$
A_x:=\bset{y\in B_{j+1}(k),\; S_{a_j}x\cap S_{a_{j+1}}y\neq\emptyset}
$$
contains at least two elements. Note that $A_x$ may contain at most four elements, for if $\zeta\in S_{a_j}x\cap S_{a_{j+1}}y\neq\emptyset$ then there exists $z\in S_{a_j}$ and $w\in S_{a_{j+1}}$ such that $\zeta=T_zx=T_wy$, meaning that
$$
S_{a_j}x\cap S_{a_{j+1}}y=S_{a_j}T_{w-z}y\cap S_{a_{j+1}}y=\bb{S_{a_{j+1}}\cap S_{a_j}\bb{w-z}}y=\bb{S_{a_{j+1}}\bb{z-w}\cap S_{a_j}}x.
$$
In particular, there exists a rectangle $R=S_{a_{j+1}}\bb{z-w}\cap S_{a_j}$ such that $Rx=S_{a_j}x\cap S_{a_{j+1}}y$. As $S_{a_j}$ and $S_{a_j}\bb{z-w}$ are aligned squares, their intersection will give us a rectangle. Since $a_j<a_{j+1}$ and the squares $T_zS_{a_j}$ and $S_{a_{j+1}}$ are aligned, the rectangle $R$ must contain at least one of the corners of $S_{a_j}$ (see Figure \ref{fig:nestedTowersLem}).

\begin{figure}[!ht]
\centering
  \begin{center}
    \includegraphics[width=0.3\textwidth]{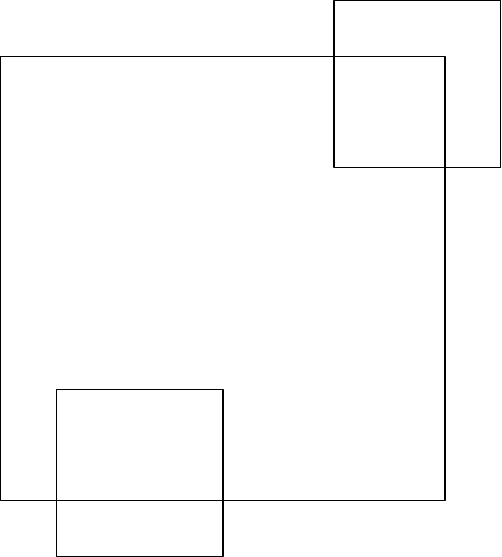}
  \end{center}
  \caption{An intersection of aligned squares is a rectangle, so that at least one of its corners belong to the smaller square in the intersection.}
  \label{fig:nestedTowersLem}
  \end{figure}

In addition, because $B_{j+1}(k)\subset B_{j+1}(0)$ it fulfills property $(F_1)$ of an $S_{a_{j+1}}$-set, and so each element of the set of `corners' of $S_{a_j}x$:
$$
\bset{T_{a_j\bb{1+i}}x,\;T_{a_j\bb{1-i}}x,\;T_{a_j\bb{-1+i}}x,\;T_{a_j\bb{-1-i}}x}
$$
belongs to the set $S_{a_{j+1}}y$ for a unique $y\in A_x$. We conclude that $A_x$ cannot contain more than four elements.\\
Next, note that $R$ is a closed rectangle as an intersection of two closed squares. We get that
$$
S_{a_j}x=\uplus\; R_\alpha x,
$$
where the collection $\bset{R_\alpha}$ contains at most four disjoint closed rectangles, and the union is disjoint since $x\in B_j(k)$ for which property $(F_1)$ of an $S_{a_j}$-set holds. This yields that $S_{a_j}=\uplus\; R_\alpha$, which is a contradiction to the fact that a square is a connected set, and thus concludes the proof of the claim.

We will prove the measurability of $B_j(k)$ by induction on $k$.  Recall that an $S$- set $B\subset X$ satisfies condition $(F_2)$ if for every $B'\subset B$ measurable, and every $A\subset S$ measurable the set $AB':=\underset{z\in A}\bigcup T_zB'$ is a measurable subset of $X$. Now, for every $j$ the set $B_j(0)$ is measurable by property $(F_2)$. Assume that for every $j$ we know that $B_j(k)$ is measurable. Following the claim above and the definition of $B_j(k+1)$, for $x\in B_j(k)$:
\begin{eqnarray*}
x\in B_j(k+1)&\overset{\text{def}}\iff&S_{a_j}x\subset S_{a_{j+1}}B_{j+1}(k)\overset{\text{claim}}\iff\exists y\in B_{j+1}(k), S_{a_j}x\subset S_{a_{j+1}}y\\
&\iff& \exists y\in B_{j+1}(k), x\in S_{a_{j+1}-a_j}y\iff x\in S_{a_{j+1}-a_j}B_{j+1}(k).
\end{eqnarray*}
We conclude that 
$$
B_j(k+1)=B_j(k)\cap S_{a_{j+1}-a_j}B_{j+1}(k),
$$
which is measurable as intersection of two measurable sets, since property $(F_2)$ holds for $B_\ell(k)$ for every $\ell$, by the induction assumption.

It is left to show that this sequence saturates the whole space, namely that $\mu\bb{S_{a_n}B_n}\nearrow1$. Following the claim above if $x\in B_j(k)\setminus B_j(k+1)$, and $y\in B_{j+1}(k-1)$, is such that $T_zx=y$, then necessarily $y\nin B_{j+1}(k)$. We obtain that
\begin{eqnarray}\label{eq:NestedErrors}
S_{a_j}B_j(k)\setminus S_{a_j}B_j(k+1)\subseteq S_{a_{j+1}}B_{j+1}(k-1)\setminus S_{a_{j+1}}B_{j+1}(k).
\end{eqnarray}
In addition, if $S_{a_j}x\cap S_{a_{j+1}-2a_j}B_{j+1}(0)\neq\emptyset$, then there exists $z\in S_{a_j}$ and $w\in S_{a_{j+1}-2a_j}$ such that $T_{z-w}x\in B_{j+1}(0)$, but then for every $\xi\in S_{a_j}$ we have that $\xi+w-z\in S_{a_{j+1}}$ and so:
$$
T_\xi x=T_{\xi+w-z}T_{z-w}x\in S_{a_{j+1}}B_{j+1}(0)\Rightarrow S_{a_j}x\subset S_{a_{j+1}}B_{j+1}(0),
$$
contradicting the fact that $x\nin B_j(1)$. We conclude that
\begin{eqnarray}\label{eq:BndOnErrors}
S_{a_j}B_j(0)\setminus S_{a_j}B_j(1)\subseteq X\setminus S_{a_{j+1}-2a_j}B_{j+1}(0).
\end{eqnarray}

Combining (\ref{eq:NestedErrors}) and (\ref{eq:BndOnErrors}) one can see that:
\begin{eqnarray*}
S_{a_j}B_j(k)\setminus S_{a_j}B_j(k+1)&\overset{(\ref{eq:NestedErrors})}\subseteq& S_{a_{j+1}}B_{j+1}(k-1)\setminus S_{a_{j+1}}B_{j+1}(k)\overset{(\ref{eq:NestedErrors})}\subseteq\cdots\\
&\overset{(\ref{eq:NestedErrors})}\subseteq& S_{a_{j+k}}B_{j+k}(0)\setminus S_{a_{j+k}}B_{j+k}(1)\overset{(\ref{eq:BndOnErrors})}\subseteq X\setminus S_{a_{k+j+1}-2a_{k+j}}B_{k+j+1}(0).\nonumber
\end{eqnarray*}
Now, using remark \ref{rmk:JonsObservation} for every $k$ and $j$ we get that for $m=j+k$,
\begin{eqnarray}\label{eq:measureBnd}
\nonumber\mu\bb{S_{a_j}B_j(k)\setminus S_{a_j}B_j(k+1)}&\le&\mu\bb{X\setminus S_{a_{m+1}-2a_{m}}B_{m+1}(0)}=1-\mu\bb{S_{a_{m+1}-2a_{m}}B_{m+1}(0)}\\
&=&1-m\bb{S_{a_{m+1}-2a_{m}}}\cdot\frac{\mu\bb{S_{a_{m+1}}B_{m+1}(0)}}{m\bb{S_{a_{m+1}}}}\\
\nonumber&\le& 1-\frac{\bb{a_{m+1}-2a_m}^2}{a_{m+1}^2}\bb{1-\eps_m}<2\eps_m+\frac{4a_m}{a_{m+1}}.
\end{eqnarray}
We note that by the triangle inequality:
\begin{eqnarray*}
&&\mu\bb{X\setminus S_{a_j}B_j(n)}\le \mu\bb{X\setminus S_{a_j}B_j(0)}+\sumit k 0 {n-1}\mu\bb{S_{a_j}B_j(k)\setminus S_{a_j}B_j(k+1)}\\
&\overset{(\ref{eq:measureBnd})}\le& \mu\bb{X\setminus S_{a_j}B_j(0)}+\sumit k 1 {n-1}\bb{2\eps_{j+k}+\frac{4a_{j+k}}{a_{j+k+1}}}<2\sumit k j \infty \bb{\eps_k+\frac{2a_k}{a_{k+1}}}.
\end{eqnarray*}
Since the latter is the tail of a converging series, it tends to 0, concluding the proof of $(N_3)$, as $\mu(X\setminus B_n)=\limit k\infty \mu(X\setminus B_n(k))\underset{n\rightarrow\infty}\longrightarrow 0$.
\end{proof}

\bibliographystyle{plain}
 {\bibliography{Weiss_extension}}
\end{document}